\pdfoutput=1
\documentclass[11pt,a4paper]{amsart}
\usepackage[margin=1in]{geometry}
\usepackage[utf8]{inputenc}
\usepackage[T1]{fontenc}
\usepackage[dvipsnames]{xcolor}
\usepackage{graphicx}
\usepackage{microtype}
\usepackage{fnpct}

\usepackage{amsmath}
\usepackage{amssymb}
\usepackage{eucal}
\usepackage{mathrsfs}
\usepackage{tikz-cd}
\renewcommand{\injlim}{\varinjlim}
\renewcommand{\projlim}{\varprojlim}

\usepackage{enumitem}
\usepackage{booktabs}
\usepackage[pdfusetitle,colorlinks]{hyperref}
\hypersetup{bookmarksdepth=2,pdfencoding=unicode,allcolors=MidnightBlue}
\usepackage[capitalise,noabbrev,nosort]{cleveref}

\crefname{equation}{}{}
\crefformat{enumi}{(#2#1#3)}
\crefmultiformat{enumi}{(#2#1#3)}{ and~(#2#1#3)}{, (#2#1#3)}{, and~(#2#1#3)}
\crefname{enumi}{}{}
\newlist{conenum}{enumerate}{1}
\setlist[conenum,1]{label=(\roman*),ref=\roman*}
\creflabelformat{conenumi}{(#2#1#3)}
\crefname{conenumi}{}{}

\numberwithin{equation}{section}
\theoremstyle{plain}
\newtheorem{Theorem}{Theorem}

\crefname{Theorem}{Theorem}{Theorems}
\newtheorem{conjecture}[equation]{Conjecture}
\crefname{conjecture}{Conjecture}{Conjectures}
\newtheorem{corollary}[equation]{Corollary}
\newtheorem{lemma}[equation]{Lemma}
\newtheorem{proposition}[equation]{Proposition}
\newtheorem{theorem}[equation]{Theorem}

\theoremstyle{definition}
\newtheorem{assumption}[equation]{Assumption}
\newtheorem{definition}[equation]{Definition}
\newtheorem{example}[equation]{Example}

\theoremstyle{remark}

\newtheorem{remark}[equation]{Remark}

\let\oldAA\AA\let\AA\relax
\let\oldVec\Vec\let\Vec\relax

\newcommand{\FF}{\mathbf{F}}
\newcommand{\NN}{\mathbf{N}}
\newcommand{\ZZ}{\mathbf{Z}}
\newcommand{\QQ}{\mathbf{Q}}
\newcommand{\RR}{\mathbf{R}}
\newcommand{\CC}{\mathbf{C}}
\newcommand{\HH}{\mathbf{H}}

\newcommand{\AA}{\mathbb{A}}
\newcommand{\GG}{\mathbb{G}}

\newcommand{\E}{\mathrm{E}}
\newcommand{\F}{\mathrm{F}}

\newcommand{\Coh}{\operatorname{Coh}}
\newcommand{\Db}{\operatorname{D}^{\textnormal{b}}}
\newcommand{\D}{\operatorname{D}}
\newcommand{\Fun}{\operatorname{Fun}}

\newcommand{\Ind}{\operatorname{Ind}}
\newcommand{\KH}{\operatorname{KH}}

\newcommand{\KSp}{\operatorname{KSp}}
\newcommand{\KU}{\operatorname{KU}}
\newcommand{\Map}{\operatorname{Map}}
\newcommand{\Mat}{\operatorname{Mat}}
\newcommand{\NCoh}{\operatorname{NCoh}}
\newcommand{\NK}{\operatorname{NK}}
\newcommand{\NVec}{\operatorname{NVec}}
\newcommand{\PShv}{\operatorname{PShv}}
\newcommand{\Perf}{\operatorname{Perf}}
\newcommand{\Pro}{\operatorname{Pro}}
\newcommand{\REnd}{\operatorname{REnd}}
\newcommand{\Shv}{\operatorname{Shv}}
\newcommand{\Sh}{\operatorname{Sh}}
\newcommand{\Spec}{\operatorname{Spec}}
\newcommand{\Tor}{\operatorname{Tor}}
\newcommand{\Vec}{\operatorname{Vec}}

\newcommand{\fib}{\operatorname{fib}}

\newcommand{\id}{\operatorname{id}}
\newcommand{\ko}{\operatorname{ko}}
\newcommand{\kr}{\operatorname{kr}}
\newcommand{\ksp}{\operatorname{ksp}}
\newcommand{\ku}{\operatorname{ku}}
\newcommand{\map}{\operatorname{map}}
\newcommand{\op}{\operatorname{op}}
\newcommand{\sw}{\operatorname{sw}}

\newcommand{\wt}{\operatorname{wt}}

\newcommand{\X}{\mathord{-}}

\newcommand{\llparenthesis}{(\!(}
\newcommand{\rrparenthesis}{)\!)}

\newcommand{\cat}[1]{\mathcal{#1}}
\newcommand{\idl}[1]{\mathfrak{#1}}
\newcommand{\cst}[1]{\underline{#1}}
\newcommand{\Cls}[1]{\mathscr{#1}}
\newcommand{\Cat}[1]{\mathsf{#1}}

\title[\texorpdfstring{\(K(\Cls{C}(X))\)}{K(C(X))}]{\texorpdfstring{\(K\)}{K}-theory of rings of continuous functions}
\author{Ko Aoki}
\address{Max Planck Institute for Mathematics,
  Vivatsgasse 7, 53111 Bonn, Germany
}
\email{aoki@mpim-bonn.mpg.de}
\date{\today}

\begin{document}

\begin{abstract}
  We study
  the algebraic \(K\)-theory
  of the ring of continuous functions
  on a compact Hausdorff space
  with values in a local division ring,
  e.g., a local field:
  We compute its negative \(K\)-theory
  and show its \(K\)-regularity.
  The complex case reproves the results
  of Rosenberg, Friedlander--Walker, and Cortiñas--Thom.
  Our consideration in the real case proves
  two previously unconfirmed claims made by Rosenberg in 1990.
  The algebraic nature of our methods
  enables us to
  deal with the nonarchimedean
  and noncommutative cases analogously.
\end{abstract}

\maketitle
\setcounter{tocdepth}{1}
\tableofcontents

\section{Introduction}\label{s:intro}

Grothendieck chose the letter ``K'' over ``C''
for \(K\)-theory to distinguish it from
the ring of continuous functions~\(\Cls{C}\);
see~\cite{Bak87}.
This is indeed practical;
e.g.,
a typical course on \(K\)-theory
starts with
the observation that \(K_{0}(\Cls{C}(X;\CC))\)
computes \(\KU^0(X)\)
for a compactum\footnote{In this paper,
  we use the term ``compactum'' to mean a compact Hausdorff space.
}~\(X\).
In this paper,
we study the (nonconnective) algebraic \(K\)-theory of~\(\Cls{C}(X;F)\)
for a local division ring (e.g., a local field)~\(F\).

The case \(F=\CC\) has been studied since
Rosenberg's influential paper~\cite{Rosenberg90}:
Friedlander--Walker~\cite[Theorem~5.1]{FriedlanderWalker01C}
first showed
that \(K_{*}(\Cls{C}(D^{n};\CC))\) vanishes for \({*}<0\)
and
then
Cortiñas--Thom~\cite[Corollary~6.9]{CortinasThom12}
in general
proved that \(K_*(\Cls{C}(X;\CC))
\simeq
{\ku}^{-*}(X)\) holds for \({*}\leq0\)
by completing Rosenberg's argument.
Moreover,
as pointed out in~\cite[Remark~8.2]{CortinasThom12},
it follows from~\cite{Rosenberg90,FriedlanderWalker01C}
that the morphism
\(K(\Cls{C}(X;\CC))\to K(\Cls{C}(X;\CC)[T_{1},\dotsc,T_{n}])\)
is an equivalence for \(n\geq0\).
Our main results
generalize those theorems to general~\(F\):

\begin{Theorem}\label{main_a}
  For a compactum~\(X\),
  we have canonical isomorphisms
  \begin{align*}
    K_*(\Cls{C}(X;\RR))
    &\simeq
    {\ko}^{-*}(X),
    &
    K_*(\Cls{C}(X;\CC))
    &\simeq
    {\ku}^{-*}(X),
    &
    K_*(\Cls{C}(X;\HH))
    &\simeq
    {\ksp}^{-*}(X)
  \end{align*}
  for \({*}\leq0\).
  Here the right-hand sides are
  the sheaf cohomology groups.
\end{Theorem}

\begin{Theorem}\label{main_na}
  For a profinite set (aka totally disconnected compactum)~\(X\)
  and
  a local division ring~\(F\),
  the negative \(K\)-theory of
  \(\Cls{C}(X;F)\) vanishes.
\end{Theorem}

\begin{Theorem}\label{main_hi}
  For a compactum~\(X\)
  and a local division ring~\(F\),
  the ring
  \(\Cls{C}(X;F)\) is \(K\)-regular;
  i.e.,
  the tautological map
  \(K(\Cls{C}(X;F))\to K(\Cls{C}(X;F)[T_{1},\dotsc,T_{n}])\)
  is an equivalence for \(n\geq0\).
\end{Theorem}

Several remarks are in order:

\begin{remark}\label{x4ajpq}
  In \cref{main_a},
  we cannot compute the right-hand side via
  the Atiyah--Hirzebruch spectral sequence;
  see \cref{xqrboy}.
\end{remark}

\begin{remark}\label{xszqjm}
  Our proof shows that
  the isomorphisms in \cref{main_a}
  come from corresponding equivalences
  between \(0\)-truncated
  commutative (or associative in the case of~\(\HH\)) ring spectra.
\end{remark}

\begin{remark}\label{x97um4}
  Rosenberg~\cite[Theorem~2.3]{Rosenberg97}
  considered
  the variant of \cref{main_a} for real (i.e., \(C_{2}\)-) compacta.
  We treat this variant in \cref{ss:kr}.
\end{remark}

\begin{remark}\label{x0p41o}
  Cortiñas--Thom~\cite[Theorem~1.5]{CortinasThom12}
  proved that
  \(A\otimes_{\CC}\Cls{C}(X;\CC)\)
  is \(K\)-regular for a smooth \(\CC\)-algebra~\(A\).
  Even in the complex case, our argument gives us
  a stronger \(K\)-regularity result than theirs:
  We show that
  any smooth \(\Cls{C}(X;\CC)\)-algebra
  is \(K\)-regular in general;
  see \cref{x8biir}.
\end{remark}

\begin{remark}\label{252df46ce5}
In the nonarchimedean case,
  similar statements hold
  for rings of integers;
  e.g.,
  for a profinite set~\(X\)
  and a prime~\(p\),
  the ring \(\Cls{C}(X;\ZZ_{p})\)
  is \(K\)-regular
  and
  has vanishing negative \(K\)-theory.
  We treat this version in \cref{ss:order}.
\end{remark}

Having stated our main theorems,
we point out that
in this paper we establish
two previously unproven statements
Rosenberg made in 1990:
\begin{enumerate}
  \item
    In~\cite[Theorems~2.4]{Rosenberg90},
    Rosenberg stated the real and complex cases of \cref{main_a}.
    The complex case was proven by
    Cortiñas--Thom~\cite{CortinasThom12}.
    Our proof in the complex case is different from theirs;
    see \cref{ss:methods}.
    Rosenberg later considered
    a variant for real (i.e., \(C_{2}\)-) compacta;
    see~\cite[Theorem~2.3]{Rosenberg97}.
    We also treat this variant in \cref{ss:kr}.
  \item
    In~\cite[Theorems~2.5\,(2)]{Rosenberg90},
    Rosenberg claimed that
    for a metrizable contractible compactum~\(X\),
    for nonnegative integers \(m\) and~\(n\),
    any finitely generated projective modules over
    \(\Cls{C}(X;\RR)[\NN^m\times\ZZ^n]\) is free.
    The complex case was proven
    by Cortiñas--Thom~\cite[Theorem~6.3]{CortinasThom12}.
    Our proof here (see \cref{xla9jq})
    just reuses their argument
    with an additional input,
    which is first proven in this paper.
\end{enumerate}

In the rest of this section,
we describe our methods in \cref{ss:methods},
illustrate applications in algebraic \(K\)-theory in \cref{ss:k_th},
and explain the connection between this work
and our ongoing work in \cref{ss:future}.

\subsection{Methods}\label{ss:methods}

We first review
the original approach
of Rosenberg~\cite{Rosenberg90}
to this problem
over the complex (or real) numbers.
It is unstable and geometric:
Suppose that we want to show that
\(K_{*}(\Cls{C}(X))\to K_{*}(\Cls{C}(X\times[0,1]))\)
is an isomorphism for \({*}\leq0\).
First,
by Bass delooping,
it suffices to show that
\(K_{0}(\Cls{C}(X)[M])\to K_{0}(\Cls{C}(X\times[0,1])[M])\)
is an isomorphism for the monoid \(M=\ZZ^{-{*}}\).
Then we consider an unstable generalization of this question,
asking whether two idempotent matrices
over \(\Cls{C}(X\times[0,1])[M]\)
are conjugate
given that they are conjugate over \(\Cls{C}(X\times\{0\})[M]\).
An idempotent matrix over \(\Cls{C}(X)[M]\)
can be regarded as a continuous map
\(X\to\Mat_{\infty}(\CC[M])\)
when we equip a suitable topology on the target
and we can interpret this problem as a homotopy lifting problem.
Hence it suffices to show
that certain maps of topological spaces
are fibrations in a suitable sense.
It is subtle to
study these infinite-dimensional spaces,
but it was achieved
by Cortiñas--Thom~\cite{CortinasThom12}.
One key input is the fact that
\(\ell^{\infty}(\NN;\CC)\) is \(\Ind\)-smooth,
which they proved using the resolution of singularities
in~\cite[Theorem~5.7]{CortinasThom12}.
They called their method \emph{algebraic approximation}:
Like any \(\CC\)-algebra,
\(\Cls{C}(X;\CC)\) is
a filtered colimit of an algebra of finite type.
Then we can interpret
a map
\(A\to\Cls{C}(X;\CC)\) as a map \(X\to(\Spec A)(\CC)\)
of topological spaces.
As pointed out in~\cite[Remark~5.8]{CortinasThom12},
the real case has an extra complication
as the topology of real points
(or any \(F\)-points for
local fields~\(F\neq\CC\)) does not behave
as nicely as in the complex case.
We overcome this problem
by finding site-theoretic explanations behind this
type of the argument:
We see that the cdh~topology of varieties
behaves nicely with topology of \(F\)-points
for any local field~\(F\).
As a consequence,
we show the following,
which fills in the last piece
of Cortiñas--Thom's argument in the real case:

\begin{Theorem}\label{linf}
  The \(\RR\)-algebra \(\ell^{\infty}(\NN;\RR)\) is \(\Ind\)-smooth.
\end{Theorem}

In fact,
we can characterize compacta~\(X\)
such that \(\Cls{C}(X;\RR)\)
is \(\Ind\)-smooth; see \cref{dig}.

Therefore,
we can also get the real case of \cref{main_a,main_hi}
by using their argument.
Moreover,
by using the argument in~\cite[Theorem~6.3]{CortinasThom12},
we obtain the following unstable result
which our main approach cannot prove:

\begin{corollary}\label{xla9jq}
  Let \(X\) be a contractible compactum
  and \(M\) be a commutative monoid that
  is countable,
  torsion-free, seminormal\footnote{A commutative monoid
    is called \emph{seminormal} if \(a^2=b^3\) implies
    the existence of~\(c\)
    with \(a=c^{3}\) and \(b=c^{2}\).
  }, and cancellative.
  Then every finitely generated projective module
  over \(\Cls{C}(X;\RR)[M]\) is free.
\end{corollary}

Whereas this fibration approach
gives us some unstable results,
it has
some difficulties:
\begin{itemize}
  \item
    We need to understand the unstable problem first:
    For example,
    we need to know that \cref{xla9jq} is true
    for \(X={*}\) and \(M=\ZZ^{-{*}}\),
    which is a deep result
    built upon the resolution of
    Serre's problem due to Quillen and Suslin.
  \item
    Since the result like \cref{linf} depends on the resolution of singularities,
    it does not work in the characteristic~\(p\) situation
    such as \(F=\FF_{p}\llparenthesis T\rrparenthesis\).
  \item
    It does not imply
    our strong \(K\)-regularity result
    explained in \cref{x0p41o}.
\end{itemize}
Our proof does not use Hironaka's theorem.
It also does not use Quillen--Suslin's theorem
as we do not use Bass delooping.
Our main observation is
that our site-theoretic study
of the topology is enough to
conclude the main results
up to recent theorems on \(K\)-theory.
On the topological side,
one crucial, yet trivial, observation
is the following nature of cohomology:

\begin{proposition}\label{zero_zero}
  Let \(X\) be a compactum
  and \(E\) a connective spectrum.
  Then for any class \(\alpha\in E^{-*}(X)\)
  for \({*}<0\),
  there is a finite cover by closed subsets
  \(Z_{1}\cup\dotsb\cup Z_{n}=X\)
  such that
  \(E^{-*}(X)\to E^{-*}(Z_{k})\) kills~\(\alpha\) for each~\(k\).
\end{proposition}

\begin{proof}
  By \cref{approx},
  \(\alpha\) is in the image of
  \(E^*(P)\to E^*(X)\)
  for some map to a polyhedron~\(P\).
  Hence we can assume \(X=P\).
  Then the cover by closed cells satisfies this property
  as for each cell~\(C\) we have \(E^{-*}(C)=0\).
\end{proof}

Basically,
our strategy is to prove
a similar result for \(K(\Cls{C}(\X;F))\)
in a finer topology,
which we call the cd~topology.
This use of topology is inspired by
the resolution of Weibel's conjecture
on negative \(K\)-theory\footnote{In characteristic~\(0\),
  this was solved by
  Cortiñas--Haesemeyer--Schlichting--Weibel~\cite{CHSW08}
  before.
} by Kerz--Strunk--Tamme~\cite{KerzStrunkTamme18}.
We use their results to prove the main theorem.

\subsection{Application: examples in \texorpdfstring{\(K\)}{K}-theory}\label{ss:k_th}

One reason we study
rings of the form \(\Cls{C}(X;\RR)\)
is that they give pathological examples
in commutative algebra.
We construct commutative rings
with interesting \(K\)-theoretic properties.

Consider a noetherian commutative ring.
Its weak dimension (aka flat dimension)
coincides with
its global dimension (aka projective dimension).
If it is finite then
the ring is regular
and hence has vanishing negative \(K\)-theory.
Our construction for the following is not related
to rings of continuous functions,
but
it is included in this paper
as it appears to be new:

\begin{Theorem}\label{xpogpw}
  Over any field,
  there is a commutative ring
  of weak dimension~\(\leq1\)
  and of global dimension~\(\leq2\)
  whose \(K_{-1}\) does not vanish.
\end{Theorem}

One could still expect
some relations between the vanishing of negative \(K\)-theory
and homological dimensions.
For example, the example we construct
for \cref{xpogpw},
its \(K\)-theory is \((-1)\)-connective.
However,
using rings of continuous functions,
we obtain the following:

\begin{Theorem}\label{2cd33f5deb}
  There is a commutative ring
  of weak dimension~\(\leq1\)
  whose \(K\)-theory is not bounded below.
  Under the continuum hypothesis,
  there is such a ring
  with global dimension~\(\leq3\).
\end{Theorem}

\subsection{Future plans}\label{ss:future}

Cortiñas--Thom's
homotopy invariance theorem
is the commutative complex case
of~\cite[Conjecture~2.2]{Rosenberg97}:

\begin{conjecture}[Rosenberg]\label{288b87f0e9}
  For any (real) C*-algebra\footnote{Here it suffices to consider unital C*-algebras
    since the unital case implies the nonunital case.
  }~\(A\),
the tautological map
  \(K_{*}(A)\to K_{*}(\Cls{C}([0,1];A))\)
  is an isomorphism for \({*}\leq0\).
\end{conjecture}

Note that
\cref{main_kr} settles this conjecture
in the commutative case.
We will reinterpret this type of statement
using condensed mathematics (see~\cite{Condensed})
in our forthcoming paper~\cite{k-ros-2}
and
will give a proof
of \cref{288b87f0e9}
for~\(K_{-1}\)
in a separate paper.

We also plan to study
the motivic cohomology
of \(\Cls{C}(X;\CC)\)
using the techniques developed in this paper.

\subsection*{Organization}\label{ss:outline}

The first part of this paper is about general topology
of compacta:
We review facts about compacta in \cref{s:comp},
introduce the cd~topology on the category of compacta in \cref{s:cd},
and compare it with the cdh~topology in algebraic geometry in \cref{s:cd_rh}.
In \cref{s:ind_sm},
we study geometric properties
of the ring \(\Cls{C}(X;\RR)\)
and prove \cref{linf,xpogpw,2cd33f5deb} there.
\Cref{s:exc} proves
special properties
of the rings of the form \(\Cls{C}(X;A)\)
for nice topological algebras~\(A\).
\Cref{s:nc} contains a technical material
on \(K\)-theory that is only needed
to treat the noncommutative case.
We prove \cref{main_a,main_na,main_hi} in \cref{s:main}
by combining the results obtained in \cref{s:cd_rh,s:exc,s:nc}.
See \cref{f:leitfaden}
for the dependency between those sections.

\begin{figure}[htbp]
  \begin{tikzcd}[row sep=small]
    \text{\labelcref{s:comp}}\ar[r]&
    \text{\labelcref{s:cd}}\ar[r]\ar[rd]&
    \text{\labelcref{s:cd_rh}}\ar[r]\ar[rd]&
    \text{\labelcref{s:ind_sm}}\\
    {}&
    {}&
    \text{\labelcref{s:exc}}\ar[r]&
    \text{\labelcref{s:main}}\\
    {}&
    {}&
    \text{\labelcref{s:nc}}\ar[ur]
  \end{tikzcd}
  \caption{A graph of the section dependency}\label{f:leitfaden}
\end{figure}
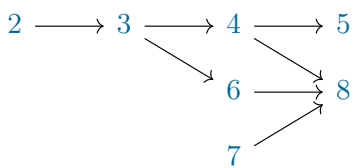

\subsection*{Acknowledgments}\label{ss:ack}

I thank
Guido Bosco,
Alexander Efimov,
Ofer Gabber,
Matthew Morrow,
Peter Scholze,
and
Andreas Thom
for beneficial communications.
I thank
the Max Planck Institute for Mathematics for its hospitality.

\subsection*{Conventions}\label{ss:kanshu}

Since
we study related problems
using condensed mathematics~\cite{Condensed}
in our forthcoming papers
(see \cref{ss:future}),
we adhere to
the terminology used in condensed mathematics;
in particular,
an \emph{anima}
is an object in the final \(\infty\)-topos~\(\Cat{Ani}\).

For a local compactum (aka locally compact Hausdorff space)~\(X\),
we write \(X^{+}\) for its one-point compactification
with \(\infty\in X^{+}\) the point at infinity.

By abuse of terminology,
we call a morphism in a site \emph{cover}
if it generates a covering sieve.

In this paper,
we deal with noncommutative rings:
In \cref{s:cd_rh,s:ind_sm},
we only consider commutative rings.
In \cref{s:exc,s:nc,s:main},
we also consider noncommutative rings.

\section{Compacta}\label{s:comp}

In this section,
we collect facts about compacta.
Some of them are not essential in this paper
but are recorded here for future use.

\subsection{Weight}\label{ss:basic}

One of the most basic cardinal invariants
for topological spaces
is the following:

\begin{definition}\label{wt}
  For a topological space (or a locale)~\(X\),
  its \emph{weight} \(\wt(X)\) is
  the smallest cardinal~\(\kappa\)
  such that
  there exists a basis of cardinality~\(\kappa\).
\end{definition}

\begin{remark}\label{908cd02220}
  Point-set topologists
  usually
  redefine the weight
  by adding~\(\aleph_{0}\).
\end{remark}

\begin{lemma}\label{x2bjoo}
  Let \(\kappa\) be an infinite cardinal.
  Consider \((X_{i})_{i\in I}\) a family of topological spaces
  satisfying \(\#I\leq\kappa\) and \(\wt(X_{i})\leq\kappa\).
  Then \(\wt\bigl(\prod_{i\in I}X_{i}\bigr)\leq\kappa\) holds.
\end{lemma}

\begin{proof}
  This follows directly from the definition
  of the product topology.
\end{proof}

\begin{example}\label{31fd5b4abc}
  For an infinite cardinal~\(\kappa\),
  we have \(\wt([0,1]^{\kappa})=\kappa\).
\end{example}

\begin{lemma}\label{x4101g}
  Let \(\kappa\) be an infinite cardinal.
  A compactum is homeomorphic to
  a subspace of \([0,1]^{\kappa}\)
  if and only if \(\wt(X)\leq\kappa\).
\end{lemma}

\begin{proof}
  The ``only if'' direction follows from \cref{31fd5b4abc},
  so we prove the converse.
  Since cozero\footnote{A subset of a topological space~\(X\)
    is called \emph{cozero}
    if it can be written as
    the inverse image of \((0,1]\)
    for some continuous function \(X\to[0,1]\).
  } sets form a basis,
  we can take a basis
  of cardinality \(\leq\kappa\)
  that consists of cozero sets.
  Then the corresponding family of functions
  gives the desired embedding.
\end{proof}

In general, weight does not behave well
with continuous surjections, but
we have the following classical result for compacta:

\begin{theorem}[Arhangelskii]\label{wt_ineq}
For a surjection \(Y\to X\) of compacta,
  \(\wt(X)\leq\wt(Y)\) holds.
\end{theorem}

We then give another characterization
of weight using polyhedrons.

\begin{definition}\label{654a594bc8}
  A \emph{polyhedron} is a compactum
  that can be obtained as a geometric realization
  of some finite simplicial complex.
\end{definition}

\begin{proposition}\label{approx}
  Let \(\kappa\) be an infinite cardinal.
  A compactum~\(X\)
  is a \(\kappa\)-small cofiltered limit of polyhedrons
  if and only if \(\wt(X)\leq\kappa\).
\end{proposition}

\begin{proof}
  The ``only if'' direction follows
  from \cref{x2bjoo},
  so
  we prove the converse.
  By \cref{x4101g},
  we have an embedding
  \(X\hookrightarrow[0,1]^{\wt(X)}\).
  Then the proof of~\cite[Theorem~X.10.1]{EilenbergSteenrod52}
  gives the desired diagram.
\end{proof}

\begin{remark}\label{not_ext}
  Let \(\cat{C}\) denote
  the full subcategory of \(\Cat{Cpt}\)
  spanned by polyhedrons.
  In \cref{approx},
  one might want to take
  the limit of~\(P\)
  where \(P\) runs over \(\cat{C}_{X/}\).
  However, this does not work:
  Let \(X\subset[0,1]\) be the Cantor set.
  We take a function \(f\colon[0,1]\to[0,1]\)
  whose zero set is~\(X\).
  Then we consider
  \(f\) and \(-f\) as
  morphisms \([0,1]\to[-1,1]\) in \(\cat{C}\).
  Then there is no morphism \(P\to[0,1]\)
  in~\(\cat{C}\) under~\(X\)
  that equates these two morphisms.
\end{remark}

\subsection{The category of compacta}\label{ss:coacc}

We consider categorical properties of \(\Cat{Cpt}\),
the category of compacta
and continuous maps.
First,
its opposite
\(\Cat{Cpt}^{\op}\) is not compactly generated
as only finite compacta
are cocompact in \(\Cat{Cpt}\).
Nevertheless,
we have the following, which was proven in~\cite{GabrielUlmer71}:

\begin{theorem}[Gabriel--Ulmer]\label{gu}
  The category \(\Cat{Cpt}^{\op}\)
  is \(\aleph_{1}\)-compactly generated.
  Its \(\aleph_{1}\)-compact generators
  are precisely compacta of countable weight.
\end{theorem}

\begin{proposition}\label{983ff95ba5}
  Let \(\kappa\) be an infinite regular cardinal.
  A compactum~\(X\) is \(\kappa\)-cocompact
  if and only if \(\wt(X)<\kappa\) holds.
\end{proposition}

\begin{proof}
  First, it can be directly
  checked when \(\kappa=\aleph_{0}\).
  We assume \(\kappa>\aleph_{0}\)
  in the rest of the proof.

  The ``if'' direction
  follows from \cref{approx}
  since polyhedrons
  are \(\aleph_{1}\)-cocompact by \cref{gu}.

  In general,
  any \(\kappa\)-compact objects
  in an \(\aleph_{1}\)-compactly generated category
  can be written as a \(\kappa\)-small
  (filtered) colimit of \(\aleph_{1}\)-compact objects.
  This observation together
  with \cref{x2bjoo} implies the ``only if'' direction.
\end{proof}

\subsection{(Strong) shape and cohomology}\label{ss:shape}

Here we clarify
the right-hand sides of the isomorphisms in \cref{main_a}.
It is simply the sheaf cohomology groups,
but they can be equivalently described using the notion of \emph{shape},
which we review here.

We write
\(\Pro(\Cat{Ani})\)
for the full subcategory of \(\Fun(\Cat{Ani},\Cat{Ani})\)
spanned by accessible left-exact functors
and call its objects \emph{\(\Pro\)-animas}.
In~\cite[Definition 7.1.6.3]{LurieHTT}
(and essentially also in~\cite[Definition~5.3.1]{ToenVezzosi}),
for an \(\infty\)-topos~\(\cat{X}\),
its \emph{shape} is
defined to be
the \(\Pro\)-anima \(p_*\circ p^*\)
for the tautological geometric morphism \(p\colon\cat{X}\to\Cat{Ani}\).
We use this to get the following:

\begin{definition}\label{xczrw9}
  We define the \emph{shape} functor
  \({\Sh}\colon\Cat{Cpt}\to\Pro(\Cat{Ani})\)
  to be the composite
  \begin{equation*}
    \Cat{Cpt}
    \xrightarrow{\Shv}
    \Cat{Top}
    \xrightarrow{\Sh}
    \Pro(\Cat{Ani}),
  \end{equation*}
  where \(\Cat{Top}\) denotes the \(\infty\)-category
  of \(\infty\)-toposes.
\end{definition}

\begin{remark}\label{7b7ffdab5d}
  The value \(\Sh(X)\)
  of the functor given in \cref{xczrw9}
  is usually called the ``strong shape'' of~\(X\)
  in shape theory.
\end{remark}

\begin{theorem}\label{86095c550d}
  The functor \(\Sh\) defined in \cref{xczrw9}
  satisfies the following property:
  \begin{enumerate}
    \item\label{i:a156a}
      Its restriction to the category of polyhedrons
      lands in
      the full subcategory of constant \(\Pro\)-animas
      and
      is equivalent to
      the homotopy type functor.
    \item\label{i:a156b}
      It preserves cofiltered limits and finite products.
    \item\label{i:a156c}
      The cohomology of \(\Sh(X)\)
      with coefficients in a spectrum~\(E\)
      is canonically equivalent to
      the sheaf cohomology of the constant sheaf \(\cst{E}\) on~\(X\).
  \end{enumerate}
\end{theorem}

\begin{proof}
  First,
  \cref{i:a156a}
  follows from~\cite[Theorem~7.1.0.1]{LurieHTT}.
  By combining~\cite[Proposition~2.11]{Hoyois18}
  with
  Lurie's proper base change theorem~\cite[Theorem~7.3.1.16]{LurieHTT},
  we obtain~\cref{i:a156b}.
  Finally,
  \cref{i:a156c} follows from the definition of~\(\Sh\).
\end{proof}

\begin{proposition}\label{d22e6fcb24}
  The functor
  \(\Sh\colon\Cat{Cpt}\to\Pro(\Cat{Ani})\)
  lands in \(\Pro(\Cat{Ani}^{\aleph_{0}})\),
  where \(\Cat{Ani}^{\aleph_{0}}\) denotes the category of compact animas.
  Moreover,
  for an infinite regular cardinal~\(\kappa\),
  the shape of a compactum of weight~\(<\kappa\)
  is \(\kappa\)-cocompact
  in the coaccessible
  \(\infty\)-category \(\Pro(\Cat{Ani}^{\aleph_{0}})\).
\end{proposition}

\begin{proof}
  Let \(\kappa\) be an infinite regular cardinal
  and \(X\) a compactum of weight~\(<\kappa\).
  By \cref{approx},
  we can write~\(X\)
  as a \(\kappa\)-small cofiltered limit
  of polyhedrons.
  Then by~\cref{i:a156a,i:a156b} of \cref{86095c550d},
  the desired result follows.
\end{proof}

We note that
in \cref{main_a},
it is important to consider
uncompleted cohomology:

\begin{example}\label{xqrboy}
Unlike the (finite) CW~case,
  for a general compactum~\(X\),
  the cohomology \(\Gamma(X;\cst{\ku})\)
  is different from
  what the Atiyah--Hirzebruch spectral sequence computes.
  This happens because
  the map
  \(\cst{\ku}
  \to
  \projlim_{i\geq0}\tau_{\leq i}\cst{\ku}\)
  of \(\Cat{Sp}\)-valued sheaves
  does not induce an equivalence on the global section
  in general:
  For example,
  when we consider
  \begin{equation*}
    X
    =
    \biggl(\coprod_{n\geq0}S^{2n}\biggr)^{+}
    =
    \projlim_{n\geq0}\biggl(\coprod_{k=0}^{n}S^{2k}\biggr)^{+},
  \end{equation*}
  the global section contains the map
  \begin{equation*}
    \bigoplus_{n\geq0}
    \Sigma^{-2n}{\ku}
    \to
    \projlim_{i\geq0}
    \bigoplus_{n\geq0}
    \Sigma^{-2n}\tau_{\leq i}{\ku},
  \end{equation*}
  which is not an equivalence,
  as a direct summand.
\end{example}

\section{The cd~topology on compacta}\label{s:cd}

We introduce a topology
on the category of compacta (of countable weight;
see \cref{xah2cq}) analogous
to the cdp~topology for proper algebraic varieties.
We introduce it and study its basic properties in \cref{ss:cd}
and study its sheaves in \cref{ss:cd_desc}.
In \cref{ss:pfin},
we consider its restriction to profinite sets.

\begin{remark}\label{xah2cq}
We work with \(\Cat{Cpt}_{\aleph_{1}}\),
  the category of compacta of countable weight,
  instead of \(\Cat{Cpt}\) for two reasons:
  We intend to avoid dealing with a large site
  and to strengthen the main result of \cref{s:cd_rh},
  which is crucial in \cref{s:ind_sm}.

  We could work with a smaller variant
  such as
  the category
  of finite-dimensional compacta of countable weight,
  which is in some sense the minimal choice:
  It is the smallest category of \(\Cat{Cpt}\)
  (or the category of topological spaces or locales)
  that is closed under finite limits
  and contains \([0,1]\).
\end{remark}

\subsection{The cd~topology}\label{ss:cd}

\begin{definition}\label{c0f0e840c9}
  We say that a collection
  \(\{Y_{i}\to X\}_{i}\)
  of morphisms in \(\Cat{Cpt}_{\aleph_{1}}\)
  is a \emph{cd~cover} (of complexity~\(\leq n\)) on~\(X\)
  if there is a filtration
  \begin{equation*}
    \emptyset
    =
    X_{0}
    \subset
    \dotsb
    \subset
    X_{n}
    =
    X
  \end{equation*}
  and maps \(p_{k}\colon X_{k}'\to X_{k}\)
  for \(k=1\), \dots,~\(n\)
  such that
  for each~\(k\),
  the map
  \(p_{k}\)
  factors through
  \(Y_{i}\to X\) for some~\(i\)
  and
  the map
  \(X_{k}'\setminus p_{k}^{-1}(X_{k-1})\to X_{k}\setminus X_{k-1}\)
  induced by \(p_{k}\) is a homeomorphism.
\end{definition}

We prove the following,
which is straightforward but tedious:

\begin{theorem}\label{cc4df108fb}
  There is a Grothendieck topology
  on \(\Cat{Cpt}_{\aleph_{1}}\)
  such that
  a sieve \(\cat{R}\subset(\Cat{Cpt}_{\aleph_{1}})_{/X}\) on~\(X\)
  is covering
  if it contains a cd~cover on~\(X\).
\end{theorem}

\begin{definition}\label{715a589303}
  We call the topology constructed in \cref{cc4df108fb}
  the \emph{cd~topology}.
\end{definition}

\begin{figure}[htbp]
  \centering
\begin{tikzpicture}
    \node[inner sep=0pt] at (0,0) {\includegraphics[width=240.9201194bp]{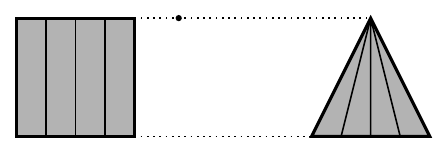}};
    \node at (0.5,0) {$\longrightarrow$};
  \end{tikzpicture}
  \caption{A typical cd~cover that is not a cover in the closed topology}\label{f:cd}
\end{figure}

\begin{proof}[Proof of \cref{cc4df108fb}]
  The only nontrivial part is the transitivity condition,
  which is \cref{2924013c0a} below.
\end{proof}

\begin{proposition}\label{2924013c0a}
  Suppose that
  \(\{Y_{i}\to X\}_{i}\) is a cd~cover
  and 
  \(\{Z_{ij}\to Y_{i}\}_{j}\) is a cd~cover
  for each \(i\in I\).
  Then the collection
  \(\{Z_{ij}\to Y_{i}\to X\}_{i,j}\)
  obtained by composing morphisms
  is again a cd~cover.
\end{proposition}

\begin{proof}[Proof of \cref{2924013c0a}]
  We proceed by induction on the length of the filtration on~\(X\).
  The case \(n=0\) is trivial.
  For \(n>0\),
  we apply \cref{two} below
  to the closed subset \(X_{n-1}\) to get the desired result.
\end{proof}

\begin{lemma}\label{two}
  Consider
  a collection
  \(\{p_{i}\colon X_{i}'\to X\}_{i}\)
  of morphisms in \(\Cat{Cpt}_{\aleph_{1}}\).
  Let \(Z\) be a closed subset of a compactum~\(X\)
  and \(U\) denote its complement.
  Then the set is a cd~cover if and only if
  the induced collection
  \(\{Z_{i}'\to Z\}_{i}\)
  and
  \(\{(U_{i}')^+\to U^+\}_{i}\)
  are cd~covers on~\(Z\) and~\(U^+\), respectively,
  where \(Z_{i}'\) and~\(U_{i}'\) denote 
  \(p_{i}^{-1}(Z)\) and \(p_{i}^{-1}(U)\), respectively.
\end{lemma}

\begin{proof}
  The ``only if'' direction is straightforward
  since
  they are cd~covers
  with respect to the filtrations
  induced by a filtration on~\(X\)
  witnessing that
  \(\{p_{i}\}_{i}\) is a cd~cover.

  We see the ``if'' direction.
  Let
  \(\emptyset=Z_{0}\subset\dotsb\subset Z_l=Z\) and
  \(\emptyset=U_{0}\subset\dotsb\subset U_m=U^+\)
  be filtrations witnessing the cd~covers.
  By refining the filtration on \(U^{+}\)
  if necessary,
  we can assume that \(\infty\in U_{1}\).
Then, with respect to the filtration
  \begin{equation*}
    X_{k}=
    \begin{cases}
      Z_{k}&\text{if \(0\leq k\leq l\),}\\
      Z\cup(U_{k-l}\setminus\{\infty\})&\text{if \(k>l\),}
    \end{cases}
  \end{equation*}
  the collection
  \(\{p_{i}\}_{i}\) is a cd~cover.
\end{proof}

\subsection{Cd~sheaves}\label{ss:cd_desc}

The sheaf condition in the cd~topology
has a nice description.

\begin{definition}\label{77e7673094}
  A pullback square
  \begin{equation}
    \label{e:blowup}
    \begin{tikzcd}
      Z'\ar[r,hook]\ar[d]&
      X'\ar[d]\\
      Z\ar[r,hook]&
      X
    \end{tikzcd}
  \end{equation}
  in \(\Cat{Cpt}\)
  is called a \emph{blowup square} if
  the bottom map
  \(Z\hookrightarrow X\) is an immersion
  and
  the map
  \(X'\setminus Z'\to X\setminus Z\)
  induced by the vertical maps
  is a homeomorphism.
\end{definition}

\begin{theorem}\label{f37c1d2136}
  Let \(\cat{C}\) be a presentable \(\infty\)-category
  and \(X\) a compactum of countable weight.
  Then a functor \(E\colon((\Cat{Cpt}_{\aleph_{1}})_{/X})^{\op}\to\cat{C}\)
  is a cd~sheaf if and only if the following are satisfied:
  \begin{conenum}
    \item
      The object
      \(E(\emptyset)\) is final.
    \item
      It carries any blowup square
      of compacta of countable weight
      to a cartesian square.
  \end{conenum}
\end{theorem}

\begin{proof}
  By Yoneda, we can assume \(\cat{C}=\Cat{S}\).
  Then
  in light of~\cite[Theorem~3.2.5]{AsokHoyoisWendt17},
  which is attributed to Voevodsky~\cite{Voevodsky10C},
  the statement follows from \cref{882d59cc9a} below.
\end{proof}

\begin{lemma}\label{882d59cc9a}
  The cd~topology on \(\Cat{Cpt}_{\aleph_{1}}\)
  is generated by
  \(\{X'\to X,Z\to X\}\)
  for blowup squares~\cref{e:blowup}
  of compacta of countable weight
  and the empty cover on~\(\emptyset\).
\end{lemma}

\begin{proof}
  In this proof,
  the \emph{cd'~topology}
  means the topology generated
  by the two types of covers in the statement.
  It is clear that the cd'~topology is coarser than the cd~topology,
  so we prove the converse.
  Consider
  a cd~covering sieve
  \(\cat{R}\) on~\(X\).
  We wish to show that
  it is also a cd'~covering sieve.
  We choose a filtration
  \begin{equation*}
    \emptyset=X_{0}
    \subset\dotsb\subset
    X_{n}=X
  \end{equation*}
  and a map
  \(p_{k}\colon X_{k}'\to X_{k}\) in~\(\cat{R}\)
  for \(k=1\), \ldots,~\(n\)
  such that
  the induced map
  \(X_{k}'\setminus p_{k}^{-1}(X_{k-1})
  \to X_{k}\setminus X_{k-1}\)
  is a homeomorphism.
  In this case,
  \(\{X_{k}'\to X_{k},X_{k-1}\to X_{k}\}\)
  generates a cd'~covering sieve.
  Hence by transitivity
  or
  the fact that
  \(\emptyset\) generates a cd'~covering sieve on \(X_{0}=\emptyset\),
  the collection
  \(\{X_{1}'\to X_{1},\dotsc,X_{n}'\to X_{n}\}\)
  also generates a cd'~covering sieve.
  Therefore, \(\cat{R}\) is a cd'~covering sieve.
\end{proof}

Cohomology is a typical example
of a cd~sheaf:

\begin{theorem}\label{f1a75190e6}
  Let \(\cat{C}\) be a presentable \(\infty\)-category.
  For an object \(C\in\cat{C}\),
  the cohomology functor
  \(\Gamma(\X;C)\colon\Cat{Cpt}_{\aleph_{1}}^{\op}\to\cat{C}\)
  is a cd~sheaf.
\end{theorem}

\begin{proof}
  As it maps~\(\emptyset\) to the final object,
  by \cref{f37c1d2136},
  this follows from \cref{cdh_sh} below.
\end{proof}

\begin{proposition}\label{cd_shv}
  For a blowup square \cref{e:blowup}
  in \(\Cat{Cpt}\),
  the square
  \begin{equation*}
    \begin{tikzcd}
      \Shv(Z')\ar[r]\ar[d]&
      \Shv(X')\ar[d]\\
      \Shv(Z)\ar[r]&
      \Shv(X)
    \end{tikzcd}
  \end{equation*}
  is cocartesian in
  the \(\infty\)-category
  of \(\infty\)-toposes \(\Cat{Top}\).
\end{proposition}

\begin{proof}
  By~\cite[Proposition~6.3.2.3]{LurieHTT},
  we need to show
  that the diagram
  obtained by considering the inverse image functors
  is cartesian.
  This claim follows from~\cite[Corollary~2.13]{ClausenJansen}.
\end{proof}

Since
the shape functor \(\Cat{Top}\to\Pro(\Cat{Ani})\) preserves colimits,
we get the following:

\begin{corollary}\label{cdh_sh}
  For a blowup square \cref{e:blowup}
  in \(\Cat{Cpt}\),
  the square
  \begin{equation*}
    \begin{tikzcd}
      \Sh(Z')\ar[r]\ar[d]&
      \Sh(X')\ar[d]\\
      \Sh(Z)\ar[r]&
      \Sh(X)
    \end{tikzcd}
  \end{equation*}
  is cocartesian in
  the \(\infty\)-category
  of \(\Pro\)-animas
  \(\Pro(\Cat{Ani})\).
\end{corollary}

\begin{example}\label{d49984addf}
  Note that
  cd~descent is stronger than what is usually called excision
  in algebraic topology.
  For example,
  singular cohomology does not determine a cd~sheaf:
  Let \(W\subset\RR^2\)
  be the Warsaw circle.
  Let
  \(D\subset\RR^2\) denote
  the union of~\(W\)
  and the bounded component of \(\RR^2\setminus W\).
  Then
  \(D\)
  and
  \(D/W\)
  are homeomorphic to~\(D^2\)
  and~\(S^2\), respectively.
  Therefore,
  \begin{equation*}
    \begin{tikzcd}
      W\ar[r]\ar[d]&
      D\ar[d]\\
      {*}\ar[r]&
      D/W
    \end{tikzcd}
  \end{equation*}
  is a blowup square,
  but the singular cohomology of~\(W\) is trivial.
\end{example}

\subsection{Variant: profinite sets}\label{ss:pfin}

Let \(\Cat{PFin}_{\aleph_{1}}\) be the category
of profinite sets of countable weight.
It is a full subcategory of \(\Cat{Cpt}_{\aleph_{1}}\)
closed under finite limits.

\begin{definition}\label{x5pwvc}
  The \emph{cd~topology} on \(\Cat{PFin}_{\aleph_{1}}\)
  is defined by restricting
  that on \(\Cat{Cpt}_{\aleph_{1}}\).
\end{definition}

The same proof as in \cref{f37c1d2136} proves the following:

\begin{theorem}\label{xds8fz}
  Let \(\cat{C}\) be a presentable \(\infty\)-category
  and \(X\) a profinite set of countable weight.
  Then a functor \(E\colon((\Cat{PFin}_{\aleph_{1}})_{/X})^{\op}\to\cat{C}\)
  is a cd~sheaf if and only if the following are satisfied:
  \begin{conenum}
    \item
      The object
      \(E(\emptyset)\) is final.
    \item
      It carries any blowup square
      of profinite sets of countable weight
      to a cartesian square.
  \end{conenum}
\end{theorem}

\section{Cdh~covers underlie cd~covers}\label{s:cd_rh}

In this section,
we relate
the cd~topology
with algebraic geometry.
For a local field~\(F\),
we write
\(\Cat{Var}_{F}\) for
the category of varieties,
i.e., separated schemes of finite type over~\(F\).
We write \(\Cat{Aff}_{F}\) for its full subcategory
spanned by affine varieties.
For a variety~\(X\),
the set of \(F\)-points
\(X(F)\) comes with a natural topology:
We use the product topology~\(F^{n}\)
for~\(\AA^{n}\)
and consider its subspace topology
for affine varieties.
In general, we get the topology by gluing.
See, e.g.,~\cite[Proposition~3.1]{Conrad12}
for details.
This gives a functor
\(\Cat{Var}_F\to\Cat{LCpt}_{\aleph_{1}}\),
where the target is the category
of local compacta of countable weight
and continuous maps.

\begin{remark}\label{xwhxfr}
The category of local compacta
  and proper continuous maps
  is copresentable
  since it is equivalent to \(\Cat{Cpt}_{*/}\).
  However, the category of local compacta
  and continuous maps
  is not;
  e.g., it does not have
  an infinite product of~\(\RR\).
  Hence the notation \(\Cat{LCpt}_{\aleph_{1}}\) here
  is not categorically justified.
\end{remark}

\begin{remark}\label{xc122x}
We can still take a minimal approach
  as in \cref{xah2cq}.
  The smallest full subcategory of topological spaces (or locales)
  that is closed under finite limits
  and contains~\(\RR\)
  is the category
  of finite-dimensional local compacta of countable weight.
\end{remark}

\begin{definition}\label{xq6j67}
  We equip \(\Cat{LCpt}_{\aleph_{1}}\) with the topology
  such that
  a sieve~\(\cat{R}\) on~\(X\) is a covering if and only if
  it contains a collection of proper maps
  \(\{X_{i}\to X\}_{i}\) such that
  \(\{X_{i}^{+}\to X^{+}\}_{i}\) is a cd~cover.
It is easy
  to deduce
  from \cref{cc4df108fb}
  that
  this is indeed a Grothendieck topology.
  We call this the \emph{cd~topology}.
\end{definition}

In this section,
we prove the following:

\begin{theorem}\label{cd_rh}
  Let \(F\) be a local field.
  We equip
  \(\Cat{Var}_F\)
  and \(\Cat{LCpt}_{\aleph_{1}}\)
  with the cdh~topology
  and cd~topologies, respectively.
  Then
  the functor
  \(\X(F)\colon\Cat{Var}_F\to\Cat{LCpt}_{\aleph_{1}}\)
  preserves covers.
\end{theorem}

This is rather immediate
when \(F=\CC\),
but otherwise it is subtle
as
a scheme surjection does not induce
a surjection on the points;
e.g., consider
\(\RR\)-points of
\(\Spec\CC\to\Spec\RR\).

\begin{remark}\label{79952efe42}
  In this paper,
  we only need the weaker variant of \cref{cd_rh},
  where ``cdh'' is replaced with ``rh''.
  This version is recorded here for completeness.
\end{remark}

\begin{proof}
  As the empty scheme underlies the empty set,
  by the definition of the cdh~topology,
  we are reduced to
  proving \cref{nis,cdp} below.
\end{proof}

\begin{proposition}\label{nis}
  Consider a cartesian square of
  varieties over a local field~\(F\)
  \begin{equation*}
    \begin{tikzcd}
      U'\ar[r,hook]\ar[d]&
      X'\ar[d]\\
      U\ar[r,hook]&
      X
    \end{tikzcd}
  \end{equation*}
  such that
  \(U\hookrightarrow X\) is an open immersion,
  \(X'\to X\) is étale,
  and
  the induced map
  \(X'\setminus U'\to X'\setminus U\)
  is an isomorphism.
  Then the set \(
  \{X'(F)\to X(F),U(F)\to X(F)\}
  \) refines to a finite closed cover of~\(X(F)\).
\end{proposition}

\begin{proposition}\label{cdp}
  Consider a cartesian square of
  varieties over a local field~\(F\)
  \begin{equation*}
    \begin{tikzcd}
      Z'\ar[r,hook]\ar[d]&
      X'\ar[d]\\
      Z\ar[r,hook]&
      X
    \end{tikzcd}
  \end{equation*}
  such that
  \(Z\hookrightarrow X\) is a closed immersion,
  \(X'\to X\) is proper,
  and
  the induced map
  \(X'\setminus Z'\to X\setminus Z\)
  is an isomorphism.
  Then the set \(
  \{X'(F)^{+}\to X(F)^{+},Z(F)^{+}\to X(F)^{+}\}
  \) is a cd~cover of~\(X(F)^{+}\).
\end{proposition}

We need the following classical facts
to prove
(and state, as the statement of \cref{cdp}
implicitly uses~\cref{i:pr} of \cref{xi1siy})
them:

\begin{lemma}\label{xi1siy}
  Let \(f\colon Y\to X\) be a map between varieties
  over a local field~\(F\).
  \begin{enumerate}
    \item\label{i:et}
      If \(f\) is étale,
      \(f(F)\colon Y(F)\to X(F)\) is a local homeomorphism.
    \item\label{i:pr}
      If \(f\) is proper,
      so is \(f(F)\colon Y(F)\to X(F)\).
  \end{enumerate}
\end{lemma}

\begin{proof}
The proof of~\cite[Proposition~4.4]{Conrad12}
  and its subsequent paragraphs
  prove \cref{i:pr,i:et}, respectively.
\end{proof}

\begin{proof}[Proof of \cref{nis}]
Let \(Z\) denote
  \(X\setminus U\simeq X'\setminus U'\).
  Then for each \(z\in Z(F)\),
  by~\cref{i:et} of \cref{xi1siy},
  we can take an open neighborhood
  \(V'_{z}\subset X'(F)\)
  that homeomorphically maps
  to \(V_{z}\subset X(F)\).
  Let \(V\) be the union of~\(V_{z}\).
  Note that
  by definition,
  the inclusion \(V\hookrightarrow X(F)\)
  factors through \(X'(F)\to X(F)\).
  As \(V\) contains~\(Z(F)\),
  \(U(F)\) and \(V\) covers~\(X(F)\).
  By the normality of \(X(F)\),
  there are
  disjoint open subsets~\(W_{1}\) and~\(W_{2}\)
  satisfying \(U(F)\cup W_{1}=V\cup W_{2}=X\).
  Then
  \(\{
    X(F)\setminus W_{1}\to X(F),
    X(F)\setminus W_{2}\to X(F)
  \}\) is a closed~cover
  of~\(X(F)\) refining our original family.
\end{proof}

\begin{proof}[Proof of \cref{cdp}]
  This is immediate from~\cref{i:pr} of \cref{xi1siy}.
\end{proof}

We have the following consequence,
which mentions neither local compacta
nor nonaffine varieties:

\begin{corollary}\label{xwb4zy}
Fix a local field~\(F\).
  Consider a functor
  \(\Cat{Aff}_{F}\to\PShv(\Cat{Cpt}_{\aleph_{1}})\)
  given by the pairing
  \((\Spec A,X)\mapsto\Map_{\Cat{CAlg}_{F}}(A,\Cls{C}(X;F))\).
  Then the right adjoint
  \(\PShv(\Cat{Cpt}_{\aleph_{1}})
  \to\PShv(\Cat{Aff}_{F})\)
  of its left Kan extension
  restricts to
  their subtoposes
  \(\Shv(\Cat{Cpt}_{\aleph_{1}})
  \to\Shv(\Cat{Aff}_{F})\)
  to determine a geometric morphism.
\end{corollary}

\section{Digression: compacta with \texorpdfstring{\(\Ind\)}{Ind}-smooth rings of continuous functions}\label{s:ind_sm}

In this section,
by ``ring'',
we mean a unital commutative ring.

This section is \emph{not} needed to prove
any of \cref{main_a,main_na,main_hi}.
\Cref{dig}, which is the main result of this section,
topologically characterizes
compacta~\(X\)
such that \(\Cls{C}(X;\RR)\) is \(\Ind\)-smooth over~\(\RR\).
We recall necessary notions
in \cref{ss:f,ss:proj,ss:ind_sm}
and prove the main theorem in \cref{ss:c_ind_sm}.
In \cref{ss:normal}, we further digress
and explain how to export pathological examples
in general topology to \(K\)-theory
using \cref{main_a};
it in particular proves \cref{2cd33f5deb}.

\begin{remark}\label{xvarmi}
  In this section, we only consider the real case for simplicity,
  but the same arguments prove the complex variants
  of our results.
\end{remark}

\begin{remark}\label{a2fc5341e6}
The nonarchimedean variant
  of our results here would be less interesting;
  what makes the real case interesting is the existence
  of an \(\F\)-space that is not totally disconnected.
  Furthermore, as we use the resolution of singularities,
  our argument does not work in positive characteristic.
\end{remark}

\subsection{\texorpdfstring{\(\F\)}{F}-spaces}\label{ss:f}

The concept of \(\F\)-spaces
was introduced in~\cite{GillmanHenriksen56}:

\begin{definition}[Gillman--Henriksen]\label{490c39f439}
  A compactum~\(X\) is called an \emph{\(\F\)-space}
  if \(\Cls{C}(X;\RR)\) is Bézout;
  i.e., every finitely generated ideal of it is principal.
\end{definition}

There are many characterizations
of \(\F\)-spaces (see, e.g.,~\cite[Theorem~14.25]{GillmanJerison60}),
but we only need the following:

\begin{theorem}[Gillman--Henriksen]\label{9dd5fcd481}
  For a compactum~\(X\),
  the following are equivalent:
  \begin{conenum}
    \item\label{i:f}
      It is an \(\F\)-space.
    \item\label{i:f_fg}
      For any continuous function \(f\colon X\to[-1,1]\),
      there is a continuous function \(g\colon X\to[-1,1]\)
      such that
      \(g\) is
      \(-1\) on \(f^{-1}[-1,0)\) and
      \(1\) on \(f^{-1}(0,1]\).
    \item\label{i:f_dom}
      The local ring of \(\Cls{C}(X;\RR)\)
      at each prime
      (or equivalently, at each maximal ideal) is a domain.
  \end{conenum}
\end{theorem}

\begin{example}\label{ac0082e6a8}
  Any extremally disconnected compactum
  is an \(\F\)-space.
  More generally,
  any basically disconnected compactum,
  i.e., a compactum whose cozero sets have open closures,
  is an \(\F\)-space.
\end{example}

\begin{example}\label{60a6bbf64f}
  Any closed subspace of an \(\F\)-space is an \(\F\)-space.
  For example, \(\beta\NN\setminus\NN\)
  is an \(\F\)-space
  that cannot be obtained from \cref{ac0082e6a8}.
\end{example}

\begin{example}[Gillman--Henriksen]\label{xgsasy}
A remarkable discovery is
  the existence of a nontrivial \emph{connected} \(\F\)-space:
  It is shown in~\cite[Example~2.8]{GillmanHenriksen56}
  that \(\beta[0,\infty)\setminus[0,\infty)\) is
  a connected \(\F\)-space.
  This cannot be obtained from \cref{60a6bbf64f}.
\end{example}

\subsection{\texorpdfstring{\(\aleph_{1}\)}{\textaleph1}-projective compacta}\label{ss:proj}

Famously,
Gleason~\cite{Gleason58} proved that
a compactum has a lifting property
against arbitrary surjections of compacta
if and only if it is extremally disconnected.
Neville--Lloyd~\cite{NevilleLloyd81}
generalized this
by considering the lifting problem
with respect to
surjections between compacta of weight \(<\kappa\)
for a fixed cardinal~\(\kappa\).
In this paper, we need the case \(\kappa=\aleph_{1}\):

\begin{theorem}[Neville--Lloyd]\label{3306b5117a}
  For a compactum~\(X\),
  the following are equivalent:
  \begin{conenum}
    \item\label{i:tdf}
      It is a totally disconnected \(\F\)-space.
    \item\label{i:a_proj}
      For an arbitrary surjection
      of compacta \(Y'\to Y\) of countable weight,
      any morphism \(X\to Y\)
      lifts to a morphism \(X\to Y'\).
  \end{conenum}
\end{theorem}

We here recall the proof
of one direction to obtain a useful corollary:

\begin{proof}[Proof of \(\text{\cref{i:a_proj}}\Rightarrow\text{\cref{i:tdf}}\)]
  Suppose that a compactum~\(X\)
  has a lifting property with respect to
  \([-1,0]\amalg[0,1]\to[-1,1]\).
  We observe that
  it must be a totally disconnected \(\F\)-space.

  We first show that \(X\) is totally disconnected
  by showing
  that two distinct points \(x_{-}\neq x_{+}\)
  lie in different connected components.
  First, we take a continuous function
  \(f\colon X\to[-1,1]\) with \(f(x_{\pm})=\pm1\).
  Then we lift~\(f\) to \(g\colon X\to[-1,0]\amalg[0,1]\).
  Then \(x_{-}\) and \(x_{+}\) belong to
  \(g^{-1}([-1,0])\) and \(g^{-1}([0,1])\),
  respectively.

  We show that \(X\) is an \(\F\)-space
  by checking~\cref{i:f_fg} of~\cref{9dd5fcd481}.
  Consider a function \(f\colon X\to[-1,1]\).
  We have a lift \(X\to[-1,0]\amalg[0,1]\).
  Then we define \(g\)
  to be its composite with
  the map
  \([-1,0]\amalg[0,1]\to\{-1,1\}\hookrightarrow[-1,1]\)
  keeping~\(\pm1\).
  This fulfills the desired condition.
\end{proof}

The argument above shows the following:

\begin{corollary}\label{e0b1f20916}
  A compactum has a lifting property
  against the tautological map \([-1,0]\amalg[0,1]\to[-1,1]\)
  if and only if
  it is a totally disconnected \(\F\)-space.
\end{corollary}

\begin{remark}\label{c63e5663c0}
  In \cref{3306b5117a},
  the lifting property in~\cref{i:a_proj}
  is not optimal;
  they proved that
  \(Y'\) can be of weight~\(\aleph_{1}\).
\end{remark}

\subsection{\texorpdfstring{\(\Ind\)}{Ind}-smooth algebras}\label{ss:ind_sm}

We use the following definition:

\begin{definition}\label{ind_sm}
  Let \(R\) be a ring.
  We say that an \(R\)-algebra
  is \emph{\(\Ind\)-smooth}
  if it can be written as a filtered colimit
  of smooth algebras over~\(R\).
\end{definition}

\begin{remark}\label{ind_sm_var}
  In contrast to the étale case,
  there is another possible definition
  of \(\Ind\)-smoothness:
  Some authors
  define
  an \(\Ind\)-smooth algebra
  to be an algebra
  that can be written as a filtered colimit
  of smooth algebras with smooth transition maps.

  We can (topologically) observe the fact that
  this is strictly stronger than our definition
  by using \cref{dig}:
  Let \(X\) be
  a totally disconnected \(\F\)-space~\(X\)
  that is not basically disconnected;
  see, e.g., \cref{60a6bbf64f}.
  Then the ring~\(\Cls{C}(X;\RR)\) is not coherent
  according to Neville's theorem~\cite{Neville90},
  which says that
  \(\Cls{C}(X;\RR)\) is coherent
  if and only if \(X\) is basically disconnected.
\end{remark}

\begin{proposition}\label{bcce4a4abc}
  Let \(R\) be a ring and \(A\) an \(R\)-algebra.
  The following are equivalent:
  \begin{conenum}
    \item\label{i:is_def}
      The \(R\)-algebra~\(A\) is \(\Ind\)-smooth.
    \item\label{i:is_f}
      Any map \(B\to A\)
      from an \(R\)-algebra of finite presentation
      factors as \(B\to B'\to A\)
      where \(B'\) is a smooth \(R\)-algebra.
  \end{conenum}
\end{proposition}

\begin{proof}
  The implication
  \(\text{\cref{i:is_def}}\Rightarrow\text{\cref{i:is_f}}\)
  is clear.
  We assume~\cref{i:is_f}
  and prove~\cref{i:is_def}.
  We consider
  the category~\(\cat{C}\)
  of \(R\)-algebras of finite presentations with maps to~\(A\).
  Then \(\cat{C}\) is filtered
  and the map \(\injlim_{C\in\cat{C}}C\to A\) is an isomorphism.
  We now consider its full subcategory \(\cat{C}'\) spanned by
  smooth \(R\)-algebras.
It suffices to show that
  the inclusion \(\cat{C}'\hookrightarrow\cat{C}\) is
  \(1\)-categorically cofinal.
  To show this,
  we fix an object \(C\in\cat{C}\)
  and
  wish to observe
  that
  the category \(\cat{D}=\cat{C}'\times_{\cat{C}}\cat{C}_{C/}\)
  is connected.
  First,
  by applying~\cref{i:is_f} to \(B=C\),
  we see that \(\cat{D}\) is nonempty.
  We then consider two objects \(C'\) and~\(C''\) of~\(\cat{D}\).
  By applying~\cref{i:is_f} to \(B=C'\otimes_C C''\),
  we see that there is an object in~\(\cat{D}\)
  receiving morphisms from~\(C'\) and~\(C''\).
  Hence \(\cat{D}\) is connected.
\end{proof}

\subsection{Compacta with \texorpdfstring{\(\Ind\)}{Ind}-smooth rings of continuous functions}\label{ss:c_ind_sm}

We prove the following:

\begin{theorem}\label{dig}
  For a compactum~\(X\),
  the following are equivalent:
  \begin{conenum}
    \item\label{i:x_tdf}
      It is a totally disconnected \(\F\)-space.
    \item\label{i:cx_is}
      The ring
      \(\Cls{C}(X;\RR)\) is 
      \(\Ind\)-smooth over~\(\RR\).
    \item\label{i:cx_sd}
      The ring
      \(\Cls{C}(X;\RR)\)
      is a filtered colimit of
      products of domains over~\(\RR\).
  \end{conenum}
\end{theorem}

As \(\beta\NN\) is a totally disconnected \(\F\)-space,
we obtain \cref{linf} as a corollary.

First note that
the implication
\(\text{\cref{i:cx_is}}\Rightarrow\text{\cref{i:cx_sd}}\)
is clear since any normal noetherian ring
is a finite product of domains.

\begin{proof}[Proof of \(\text{\cref{i:x_tdf}}\Rightarrow\text{\cref{i:cx_is}}\)]
  We consider a totally disconnected \(\F\)-space~\(X\).
  We show that \(\Cls{C}(X;\RR)\) is \(\Ind\)-smooth
  by checking~\cref{i:is_f} of~\cref{bcce4a4abc}.
  Let \(A\) be an \(\RR\)-algebra of finite type
  with a map \(A\to\Cls{C}(X;\RR)\).
We have an rh~cover
  \(\Spec A'\to\Spec A\)
  by using Hironaka's resolution of singularities
  and induction on the dimension.
  By \cref{cd_rh},
  the map \((\Spec A')(\RR)\to(\Spec A)(\RR)\)
  is a cd~cover.
Hence by \cref{3306b5117a}, we can take a lift
  \(X\to(\Spec A')(\RR)\)
  and thus we have a factorization
  \(A\to A'\to\Cls{C}(X;\RR)\).
\end{proof}

The proof of
the remaining implication needs
the following simple fact:

\begin{lemma}\label{77eb54c6ff}
  Let \(A\) be a domain.
  Then any ring map \(\ZZ[x,y]/\langle x,y\rangle\to A\)
  factors through the map
  \(\ZZ[x,y]/\langle x,y\rangle\to\ZZ[x]\times\ZZ[y]\).
\end{lemma}

\begin{proof}
  Let \(a\) and~\(b\) the images
  of \(x\) and~\(y\),
  respectively.
  These satisfy \(ab=0\),
  but since \(A\) is a domain,
  we may assume that \(a=0\)
  without loss of generality.
  Then the map factors as
  \(\ZZ[x,y]/\langle x,y\rangle
  \to\ZZ[x]\times\ZZ[y]
  \to0\times\ZZ[y]\simeq\ZZ[y]
  \xrightarrow{y\mapsto b}A\).
\end{proof}

\begin{figure}[htbp]
  \centering
  \begin{tikzpicture}
    \node[inner sep=0pt] at (0,0) {\includegraphics[width=397.38072193bp]{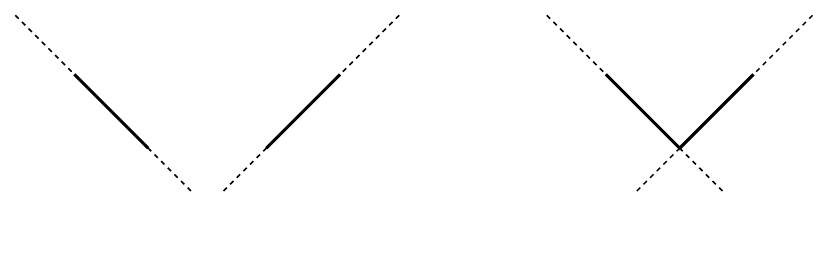}};
    \node at (-5.5,-1.75) {\(\Spec\RR[x]\)};
    \node at (-1.5,-1.75) {\(\Spec\RR[y]\)};
    \node at (4.5,-1.75) {\(\Spec\RR[x,y]/\langle xy\rangle\)};
    \node at (-6,.5) {\(-1\)};
    \node at (-4.75,-0.75) {\(0\)};
    \node at (-2.25,-0.75) {\(0\)};
    \node at (-1,.5) {\(1\)};
    \node at (3,.5) {\(-1\)};
    \node at (6,.5) {\(1\)};
    \node at (1,0) {\(\longrightarrow\)};
  \end{tikzpicture}
  \caption{The cover \([-1,0]\amalg[0,1]\to[-1,1]\)
    inside the real points of
    \(\Spec\RR[x]\amalg\Spec\RR[y]\to\Spec\RR[x,y]/\langle xy\rangle\)
  }\label{f:cross}
\end{figure}
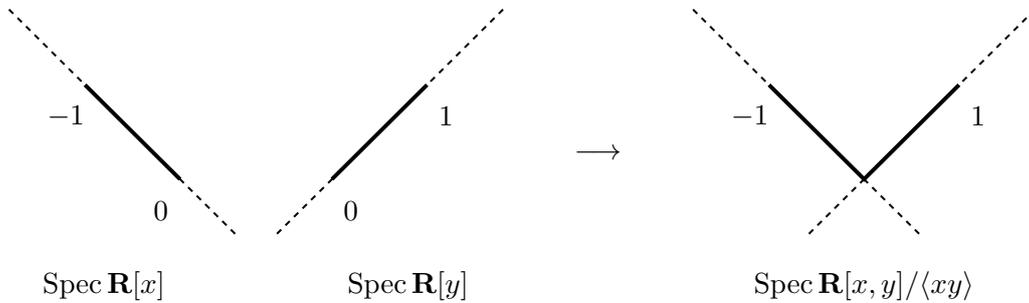

\begin{proof}[Proof of \(\text{\cref{i:cx_sd}}\Rightarrow\text{\cref{i:x_tdf}}\) of \cref{dig}]
  Let \(X\) be a compactum satisfying \cref{i:cx_sd}.
  By \cref{e0b1f20916},
  we only need to show that
  \(X\) has a lifting property
  against \([-1,0]\amalg[0,1]\to[-1,1]\).

  First, we note that
  the class of rings
  satisfying the conclusion of \cref{77eb54c6ff}
  is closed under (possibly infinite) products
  and filtered colimits.
  Hence by assumption,
  \(\Cls{C}(X;\RR)\) has this property.

  Suppose that
  we have a map \(X\to[-1,1]\).
  Then as in \cref{f:cross},
  we get a map of \(\RR\)-algebras
  \(\RR[x,y]/\langle x,y\rangle\to\Cls{C}(X;\RR)\);
  formally,
  this is determined by
  \begin{align*}
    x&\mapsto
    (X\to[-1,1]\xrightarrow{\min(0,\X)}[-1,0]),&
    y&\mapsto
    (X\to[-1,1]\xrightarrow{\max(0,\X)}[0,1]).
  \end{align*}
  From the observation above,
  we can factor this map to
  \(\RR[x]\times\RR[y]\to\Cls{C}(X;\RR)\).
  This corresponds to the desired lifting property;
  cf.~\cref{f:cross}.
\end{proof}

\subsection{Digression: rings with interesting \texorpdfstring{\(K\)}{K}-theoretic properties}\label{ss:normal}

Recall that a ring is called \emph{normal}
if its localization at each prime
(or equivalently, at each maximal ideal) is
a normal domain.
The starting point of our discussion
is the following:

\begin{proposition}\label{normal}
  Let \(X\) be a compactum.
  Then \(\Cls{C}(X;\RR)\) is normal
  if and only if
  \(X\) is an \(\F\)-space.
\end{proposition}

\begin{proof}
The ``if'' direction follows from
  the fact that
  a Bézout local domain is a valuation ring,
  which is normal.
  The ``only if'' direction follows
  from \(
  \text{\cref{i:f_dom}}
  \Rightarrow
  \text{\cref{i:f}}
  \) of \cref{9dd5fcd481}.
\end{proof}

\begin{remark}\label{e627136059}
  There is another use of ``normal''
  in the literature on rings of continuous functions:
  In~\cite{ParmenterStewart87},
  normal rings are defined to be rings such that
  \(xy=0\) implies \(1\in\ker x+\ker y\).
  By combining~\cite[Theorem~6.2]{Dyre82}
  and~\cite[Theorem~2]{ParmenterStewart87},
  we see that
  \(\Cls{C}(X;\RR)\) is normal
  if and only if
  \(X\) is an \(\F\)-space.
  Therefore,
  by \cref{normal},
  those two notions coincide for \(\Cls{C}(X;\RR)\).
\end{remark}

By combining \cref{dig,normal},
we see that
if a compactum~\(X\)
is an \(\F\)-space
that is not totally disconnected,
\(\Cls{C}(X;\RR)\)
is an example of a normal \(\RR\)-algebra
that cannot be written as a filtered colimit
of a normal \(\RR\)-algebras of finite type.
This type of an example is well known;
I was informed by Gabber about the following observation:

\begin{remark}[Gabber]\label{bce2bc0a4a}
  If a ring can be expressed
  as a filtered colimit of normal rings
  that are of finite type over a field~\(k\),
  each of its connected components
  becomes
  a filtered colimit of normal domains;
  in particular, it is a domain.
  Hence
  any connected normal \(k\)-algebra
  that is not a domain
  cannot be written as a filtered colimit
  of normal \(k\)-algebras of finite type;
  see \cref{dcd3146c40} below
  for a typical construction
  of such a \(k\)-algebra.
\end{remark}

We recall the following
example, which can be found in~\cite[Section~7]{Lazard67}:

\begin{example}[Lazard]\label{dcd3146c40}
  Over an arbitrary field~\(k\),
  we show that
  \begin{equation*}
    A
    =
    \frac
    {k[x_{i}\mid i\in[0,1]\cap\ZZ[1/2]]}
    {\langle(x_{i}-1)x_{j}\mid i<j\rangle}
  \end{equation*}
  is
  a normal algebra that is connected but not irreducible.

  It is an increasing union
  of the subalgebra \(A_{n}\subset A\)
  generated by
  the variables~\(x_{i}\)
  satisfying \(2^n i\in\ZZ\).
  Geometrically,
  \(\Spec A_{n}\) is a chain-like union of \((2^n+1)\) lines
  and
  \cref{f:bu}
  shows
  how the transition map \(\Spec A_{1}\to\Spec A_{0}\)
  (or any transition map
  \(\Spec A_{n+1}\to\Spec A_{n}\)
  locally) looks.
  We see
  that \(A\) is connected but not irreducible
  from this description.

  We now take a point~\(\idl{p}\) of \(\Spec A\)
  and
  wish to show that \(A_{\idl{p}}\) is normal.
  The point~\(\idl{p}\) determines
  a compatible family of
  points~\(\idl{p}_{n}\) of \(\Spec A_{n}\).
  If \(\idl{p}_{n}\) is a smooth point for \(n\gg0\),
  then the local ring \(A_{\idl{p}}\simeq(A_{n})_{\idl{p}_{n}}\) is normal.
  Hence we assume that
  \(\idl{p}_{n}\) is a singular point for every \(n\geq0\).
  We analyze the map \((A_{n})_{\idl{p}_{n}}\to(A_{n+1})_{\idl{p}_{n+1}}\).
  Let \(x\), \(y\), and~\(z\) denote
  the variables
  such that
  \(\idl{p}_{n}\) is the intersection of~\(L_x\) and~\(L_y\)
  and
  \(\idl{p}_{n+1}\) is the intersection of~\(L_x\) and~\(L_z\),
  where \(L_{\X}\) means the line corresponding to the variable~\(\X\);
  see \cref{f:bu}.
  Then the transition map around these points
  factors as
  \begin{equation*}
    L_x\cup L_z\to L_x\to L_x\cup L_y.
  \end{equation*}
  This shows that the map
  \((A_{n})_{\idl{p}_{n}}\to(A_{n+1})_{\idl{p}_{n+1}}\)
  factors through the local ring of~\(k[x]\)
  at \(x=0\) or \(x=1\).
  This observation shows that \(A_{\idl{p}}\)
  is the sequential colimit of the algebra of the form \(k[T]_{(T)}\)
  whose transition maps are \(\id\) or \(T\mapsto0\)
  and thus \(A_{\idl{p}}\) is isomorphic to either~\(k\) or \(k[T]_{(T)}\).
\end{example}

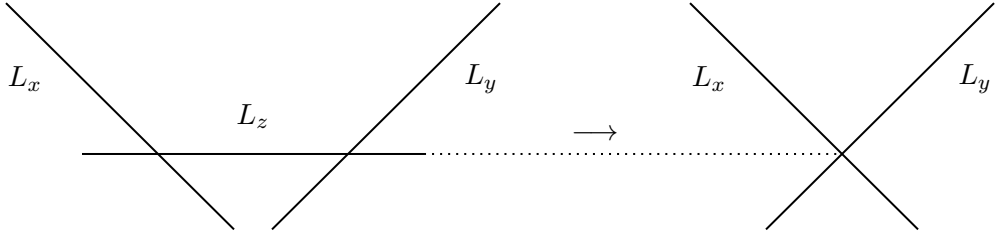
\begin{figure}[htbp]
\centering
  \begin{tikzpicture}
    \node[inner sep=0pt] at (0,0) {\includegraphics[width=383.20831539bp]{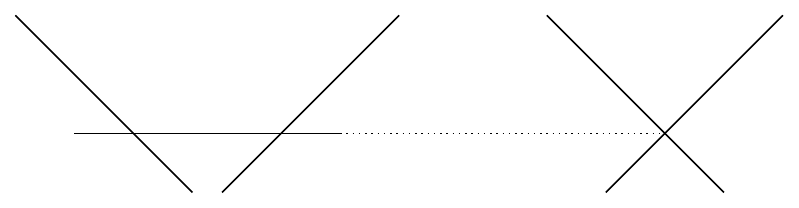}};
    \node at (1.25,-.25) {$\longrightarrow$};
    \node at (-6.25,.5) {$L_x$};
    \node at (-3.25,0) {$L_z$};
    \node at (-.25,.5) {$L_y$};
    \node at (2.75,.5) {$L_x$};
    \node at (6.25,.5) {$L_y$};
  \end{tikzpicture}
  \caption{Part of The construction in \cref{dcd3146c40}}\label{f:bu}
\end{figure}

Note that
we can modify this to obtain \cref{xpogpw}:

\begin{example}\label{5841fb8eca}
Over any field~\(k\),
  we here construct a ring of weak dimension~\(\leq1\)
  and global dimension~\(\leq2\)
  whose \(K_{-1}\) does not vanish.
  Note that any ring of weak dimension~\(\leq1\)
  is \(K\)-regular;
  see~\cite[Theorem~1.2]{BanerjeeSadhu22}.

  Consider \(B=A/\langle x_{0}-x_{1}\rangle\),
  where \(A=\injlim_{n} A_{n}\)
  is the ring constructed in \cref{dcd3146c40}.
  Note that this is a sequential colimit
  of
  the ring
  \(B_{n}=A_{n}/\langle x_{0}-x_{1}\rangle\).
  For each~\(n\geq1\),
  by excision,
  \(K_{-1}(B_{n})\simeq\ZZ\) holds and
  the map \(K_{-1}(B_{n})\to K_{-1}(B_{n+1})\) is an isomorphism.
  Hence \(K_{-1}(B)\simeq\ZZ\neq0\).

  It follows from
  the same argument as in \cref{dcd3146c40}
  that
  the local ring of~\(B\)
  is a valuation ring
  and hence \(B\) has weak dimension~\(\leq1\).
Also,
  since \(B\) is countably generated over a field,
  every ideal of~\(B\) is countably generated.
  Hence it follows from~\cite[Corollary~1.4]{Osofsky68}
  that
  \(B\) has global dimension~\(\leq2\).
\end{example}

Finally,
by using \cref{main_a},
we construct
this type of example
using general topology.
The topological input
is the following,
which was proven in~\cite[Corollary~5.2]{CalderSiegel78H}:

\begin{theorem}[Calder--Siegel]\label{05c57ea55f}
Let \(M\) be a compact topological manifold
  with boundary~\(\partial M\).
  Let \(M^{\circ}=M\setminus\partial M\) denote
  its interior.
  Let \(X\) be a CW~complex of finite type;
  i.e., it has finite cells in each dimension.
  If \(\pi_{1}(X)\) is finite,
  the map
  \begin{equation*}
    [\partial M,X]
    \to
    [\beta M^{\circ}\setminus M^{\circ},X]
  \end{equation*}
  is bijective,
  where \([\X,\X]\) denotes
  the set of homotopy classes of maps.
\end{theorem}

By applying \cref{05c57ea55f}
to \(X=K(\ZZ,i)\),
we obtain the following:

\begin{corollary}\label{x4ovub}
  In the situation of \cref{05c57ea55f},
  \begin{equation*}
    H^{i}(\partial M;\ZZ)
    \to
    H^{i}(\beta M^{\circ}\setminus M^{\circ};\ZZ)
  \end{equation*}
  is an isomorphism
  for \(i\geq2\).
\end{corollary}

\begin{remark}\label{xiab72}
  Condensed mathematics
  helps to understand statements like
  \cref{x4ovub}.
  This will be explained elsewhere.
\end{remark}

\begin{example}\label{xpw6rb}
Consider \(M=D^{n+1}\) for \(n\geq0\)
  in \cref{x4ovub}.
  Then we see that
  the map
  \(H^{i}(S^{n};\ZZ)
  \to H^{i}(\beta\RR^{n+1}\setminus\RR^{n+1};\ZZ)\)
  is an equivalence
  for \(i\geq2\).
  In particular,
  the map
  \(H^{i}(S^{n};\QQ)
  \to H^{i}(\beta\RR^{n+1}\setminus\RR^{n+1};\QQ)\)
  is an equivalence.
\end{example}

We obtain \cref{2cd33f5deb} as follows:

\begin{example}\label{f7079eb4b7}
  We use \cref{main_a}
  to construct a ring with weak dimension \(\leq1\)
  whose all negative \(K\)-theory groups are nonzero.
Consider the compactum
  \begin{equation*}
    X
    =
    \beta\biggl(\coprod_{n\geq2}\RR^{n+1}\biggr)
    \setminus
    \biggl(\coprod_{n\geq2}\RR^{n+1}\biggr).
  \end{equation*}
  It is an \(\F\)-space
  by~\cite[Theorem~2.7]{GillmanHenriksen56}.
Hence
  the weak dimension of
  \(\Cls{C}(X;\RR)\) is \(\leq1\).
For \(n\geq2\),
  since \(X\) contains \(\beta\RR^{n+1}\setminus\RR^{n+1}\)
  as a direct summand,
  by \cref{xpw6rb},
  we have \(H^n(X;\QQ)\neq0\).
  Therefore,
  we see
  \(\ko^{-{*}}(X)\otimes\QQ
  \simeq
  \bigoplus_{i\geq0}H^{4i-{*}}(X;\QQ)
  \neq0\)
  for \({*}\leq0\).
  According to \cref{main_a},
  this means that
  \(K_{*}(\Cls{C}(X;\RR))\otimes\QQ\)
  does not vanish for \({*}\leq0\).
\end{example}

Note that for examples such as \cref{f7079eb4b7},
we can also bound their global dimensions
under some set-theoretic assumptions:

\begin{remark}\label{25a0dda94b}
Let \(X\) be a local compactum.
  Assume that
  \(X\) is separable and
  \(\beta X\setminus X\) is an \(\F\)-space.
  Since \(\Cls{C}(\beta X;\RR)\)
  is a subring of \(\RR^{I}\)
  for a dense subset \(I\subset X\),
  its cardinality is \(\leq2^{\aleph_{0}}\).
  By considering a surjection
  \(\Cls{C}(\beta X;\RR)
  \to\Cls{C}(\beta X\setminus X;\RR)\),
  we observe that
  any ideal of
  \(\Cls{C}(\beta X\setminus X;\RR)\)
  is generated by \(\leq2^{\aleph_{0}}\) elements.
  For example,
  if we assume the continuum hypothesis
  \(2^{\aleph_{0}}=\aleph_{1}\),
  its global dimension is \(\leq3\)
  by~\cite[Corollary~1.4]{Osofsky68}.
  Note that
  since
  \(2^{\aleph_{0}}=\aleph_{\omega+1}\) is relatively consistent
  with the Zermelo--Fraenkel set theory
  by Cohen's theorem,
  this observation does not give a useful bound in general.
\end{remark}

\section{Excision for rings of continuous functions}\label{s:exc}

In this section,
by ``ring'',
we mean a unital associative ring.

In this section,
we prove the following:

\begin{theorem}\label{exc_a}
  Let \((A,\lVert\X\rVert)\) be a normed ring
such that there exists an \(\RR\)-algebra structure
  satisfying \(\lVert x\rVert=\lvert x\rvert\) for \(x\in\RR\).
  Let \(X\) be a compactum
  and \(R\) a commutative ring
  with a map \(R\to\Cls{C}(X;A)\).
  For an \(R\)-linear localizing invariant\footnote{We do not require that localizing invariants preserve filtered colimits.
  }~\(E\)
  valued in a presentable stable \(\infty\)-category~\(\cat{C}\),
  the functor
  \begin{equation*}
    E(\Cls{C}(\X;A))\colon
    ((\Cat{Cpt}_{\aleph_{1}})_{/X})^{\op}
    \to\cat{C}
  \end{equation*}
  is a cd~sheaf.
\end{theorem}

\begin{theorem}\label{exc_na}
Let \((A,\lVert\X\rVert)\) be a nonarchimedean discretely normed ring
  satisfying either of the following:
  \begin{conenum}
    \item\label{i:n_na}
      There is a unit~\(\pi\)
      satisfying \(\lVert\pi\rVert<1\) and
      \(\lVert\pi\rVert\lVert\pi^{-1}\rVert=1\).
    \item\label{i:n_nab}
      It is the unit disk
      of a normed ring satisfying~\cref{i:n_na}.
  \end{conenum}
  Let \(X\) be a profinite set
  and \(R\) a commutative ring
  with a map \(R\to\Cls{C}(X;A)\).
  For an \(R\)-linear localizing invariant~\(E\)
  valued in a presentable stable \(\infty\)-category~\(\cat{C}\),
  the functor
  \begin{equation*}
    E(\Cls{C}(\X;A))\colon
    ((\Cat{PFin}_{\aleph_{1}})_{/X})^{\op}
    \to\cat{C}
  \end{equation*}
  is a cd~sheaf.
\end{theorem}

We recall
general results on localizing invariants.
We first need the following notion:

\begin{definition}\label{xi411r}
We say that
  a morphism \(A\to B\) of rings
  is \emph{\(\Tor\)-unital}
  if the map \(B\otimes_{A}B\to B\) is an equivalence,
  where \(\otimes\) denotes the derived tensor product.
\end{definition}

The following
was observed in~\cite{Morrow18}:

\begin{theorem}[Morrow]\label{ed50fc2cda}
Let \(I\) be a two-sided ideal of
  a ring~\(A\).
  Then \(\ZZ\ltimes I\to\ZZ\)
  is \(\Tor\)-unital
  if and only if so is \(A\to A/I\).
\end{theorem}

We say that
a nonunital ring \(I\) satisfies \emph{excision}
(in algebraic \(K\)-theory)
if its \(K\)-theory is well defined;
formally speaking, it is the requirement that
for  any embedding \(I\hookrightarrow A\)
realizing \(I\) as a two-sided ideal,
\(K\) carries
\begin{equation*}
  \begin{tikzcd}
    \ZZ\ltimes I\ar[r]\ar[d]&
    \ZZ\ar[d]\\
    A\ar[r]&
    A/I
  \end{tikzcd}
\end{equation*}
to a cartesian square of spectra.
In that case
we can define \(K(I)\) to be
\(\fib(K(\ZZ\ltimes I)\to K(\ZZ))\)
or
\(\fib(K(A)\to K(A/I))\)
for any embedding.
Suslin proved
the following general criterion in~\cite{Suslin95}:

\begin{theorem}[Suslin]\label{3234ba150e}
  Let \(I\) be a nonunital associative ring.
  Then \(I\) satisfies excision
  if and only if \(\ZZ\ltimes I\to \ZZ\) is
  \(\Tor\)-unital.
\end{theorem}

The ``if'' direction was generalized
in~\cite[Corollary~1.4 and Remark~1.18]{LandTamme19}:

\begin{theorem}[Land--Tamme]\label{26b5462836}
  Consider a cartesian square of \(\E_{1}\)-rings
  over an \(\E_{\infty}\)-ring~\(R\)
  \begin{equation*}
    \begin{tikzcd}
      A\ar[r]\ar[d]&
      B\ar[d]\\
      A'\ar[r]&
      B'\rlap.
    \end{tikzcd}
  \end{equation*}
  Suppose that
  the induced map \(A'\otimes_A B\to B'\) of spectra
  is an equivalence;
  e.g., \(B\otimes_A B\to B\) is an equivalence\footnote{Tamme~\cite{Tamme18} proved this case before.
    We only need this case in this paper.
  }.
  Then any \(R\)-linear localizing invariant
  carries this square to a cartesian square.
\end{theorem}

\begin{example}[Suslin--Wodzicki]\label{994969f6af}
  In~\cite[Corollary~10.4]{SuslinWodzicki92},
  it is proven
  that any nonunital real Banach algebra
  with a right (or left) bounded approximate unit
  satisfies excision.
  Posing this example here is ahistorical
  as it was proven 
  before \cref{3234ba150e}.
\end{example}

\begin{remark}\label{x7gu2d}
As an application of \cref{994969f6af},
  we see that
  \(\Cls{C}_{0}(X;A)\) satisfies excision
  for a local compactum~\(X\)
  and a unital real Banach algebra~\(A\).
  This is also a corollary of our argument and \cref{ed50fc2cda}.
\end{remark}

To put \cref{exc_a,exc_na} in the context
of the excision problem,
we first need
to see that \(\Cls{C}(X;A)\to\Cls{C}(Z;A)\)
is surjective
when \(Z\hookrightarrow X\) is a closed inclusion.
First,
we recall the following from~\cite{Dugundji51}:

\begin{theorem}[Dugundji]\label{1cc406c94c}
  Let \(E\) be a locally convex real vector space
  and \(Z\hookrightarrow X\) a closed inclusion
  of metrizable topological spaces.
  Then any continuous function
  \(Z\to E\) extends to~\(X\).
\end{theorem}

\begin{corollary}\label{b485ea33ba}
  Let \(E\) be a metrizable locally convex real vector space
  and \(Z\hookrightarrow X\) an inclusion
  of compacta.
  Then any continuous function
  \(Z\to E\) extends to~\(X\).
\end{corollary}

\begin{proof}
By \cref{gu},
  we can write \(X\)
  as a limit
  of an \(\aleph_{1}\)-cofiltered family
  of compacta of countable weight \((X_{i})_{i}\).
  Let \(Z_{i}\) denote the image
  of~\(Z\) under \(X\to X_{i}\).
  Then \(Z\) is the \(\aleph_{1}\)-cofiltered
  limit of the family \((Z_{i})_{i}\)
  of compacta of countable weight.
  As the image of the function \(Z\to E\)
  is metrizable,
  it factors through
  \(Z\to Z_{i}\) for some~\(i\).
  Then we can extend \(Z_{i}\to E\) to~\(X_{i}\)
  by \cref{1cc406c94c}.
  By composing it with \(X\to X_{i}\),
  we get the desired extension.
\end{proof}

In the nonarchimedean case,
we observe the following:

\begin{proposition}\label{7d5f1b5f38}
  Let \(E\) be the underlying topological space
  of an ultrametric space.
  If \(E\) is nonempty,
for any inclusion \(Z\hookrightarrow X\) of profinite sets,
  any continuous function \(Z\to E\)
  extends to~\(X\).
\end{proposition}

\begin{proof}
  It is clear when \(Z=\emptyset\).
  Otherwise, we can replace~\(E\) with its image,
  which is a nonempty profinite set of countable weight.
  In~\cite[Theorem~2]{Halmos61},
  Halmos proved that
  nontrivial countable Boolean algebras
  are projective in the category
  of Boolean algebras.
  In light of Stone duality,
  this implies that 
  \(E\) has the desired lifting property.
\end{proof}

We then prove the \(\Tor\)-unitality.
We need the following:

\begin{lemma}\label{6dcf858e48}
  Let \(I\subset A\) be an ideal of an associative ring.
  Assume that for any
  \((b_{i})_{i=1}^{n}\in A^{n}\)
  and \((y_{i})_{i=1}^{n}\in I^{n}\)
  satisfying
  \(\sum_{i=1}^n y_{i} b_{i}=0\),
  there are
  \((a_{i})_{i=1}^n\in I^{n}\)
  and \(x\in I\)
  satisfying
  \(y_{i}=xa_{i}\) and \(\sum_{i=1}^{n}a_{i}b_{i}=0\).
  Then \(A\to A/I\) is \(\Tor\)-unital.
\end{lemma}

\begin{proof}
  We first note that
  if \(I^{2}=I\) and \(I\) is flat as a right \(A\)-module,
  then \(A\to A/I\) is \(\Tor\)-unital.
We can check those two properties
  under our assumption:
  The first one follows from the assumption
  by taking \(n=1\) and \(b_{1}=0\).
The second one follows
  from the assumption and the equational criterion of flatness.
\end{proof}

\begin{theorem}\label{xxaaq2}
  Let \((A,\lVert\X\rVert)\) be a normed algebra
  satisfying either the assumptions of~\cref{exc_a} or~\cref{exc_na}.
  Let \(Y\to X\) be a map of a compactum.
  Let \(I\) be the kernel of the map \(\Cls{C}(X;A)\to\Cls{C}(Y;A)\).
  Then \(\Cls{C}(X;A)\to\Cls{C}(X;A)/I \) is \(\Tor\)-unital.
\end{theorem}

\begin{proof}
  For \(\epsilon>0\),
  we write \(A_{<\epsilon}\)
  for the subset of~\(A\)
  consisting of elements of norm \(<\epsilon\).

  We prove that the hypothesis of \cref{6dcf858e48}
  is satisfied for \(I\subset\Cls{C}(X;A)\):
  Assume that
  \((b_{i})_{i=1}^{n}\in\Cls{C}(X;A)^{n}\)
  and \((y_{i})_{i=1}^{n}\in I^{n}\)
  satisfy
  \(\sum_{i=1}^n y_{i} b_{i}=0\).
  We wish to construct \((a_{i})_{i=1}^{n}\in I^{n}\) and~\(x\in I\)
  satisfying
  \(y_{i}=xa_{i}\) and \(\sum_{i=1}^{n}a_{i}b_{i}=0\).
  We can replace~\(Y\) by its image \(Z\subset X\).
  By enlarging~\(Z\) if necessary,
  we furthermore assume that~\(Z\)
  is the intersection of the zeros of~\(y_{i}\).

  We first suppose that \(A\) satisfies the assumption of \cref{exc_a}.
  We define a function \(x\colon X\to\RR\) by
  \begin{equation*}
    x(p)
    =
    \sqrt{\max(\lVert y_{1}(p)\rVert,\dotsc,\lVert y_{n}(p)\rVert)
    }.
  \end{equation*}
  This is continuous and vanishes on~\(Z\)
  so we regard this as an element of~\(I\).
  By assumption, it is invertible on~\(X\setminus Z\).
  For each~\(i\),
  we define a function \(a_{i}\colon X\to A\) by
  \begin{equation}
    \label{e:a_from_x}
    a_{i}(p)=
    \begin{cases}
      0&\text{if \(p\in Z\),}\\
      x(p)^{-1}y_{i}(p)&\text{if \(p\notin Z\).}
    \end{cases}
  \end{equation}
  By definition,
  each \(a_{i}\) is continuous on \(X\setminus Z\).
It is also continuous at any \(p\in Z\),
  since for~\(\epsilon>0\),
  the image of
  any element
  in \(y_{i}^{-1}(A_{<\epsilon})\)
  under~\(a_{i}\)
  has norm \(<\sqrt{\epsilon}\).
  Hence this pair is a solution to the equations.

  We then suppose that \(A\) satisfies the assumption of \cref{exc_na}.
  We first assume that \(A\) satisfies \cref{i:n_na}.
First note that
  the assumption implies that 
  \(\lVert\pi^{n}\X\rVert=\lVert\pi\rVert^{n}\lVert\X\rVert\)
  for any \(n\in\ZZ\).
  We define
  the function
  \(v\colon X\to\ZZ\cup\{\infty\}\)
  by
  \begin{equation*}
    v(p)
    =
    \biggl\lfloor
      \frac{\log_{\lVert\pi\rVert}
        \max(\lVert y_{1}(p)\rVert,\dotsc,\lVert y_{n}(p)\rVert)
      }2
    \biggr\rfloor.
  \end{equation*}
Since \(A\) is discretely valued,
  this is continuous
  when the target is equipped with the topology
  induced by \((-\infty,\infty]\).
  Hence \(x(\X)=\pi^{v(\X)}\) defines an element of~\(I\)
  and is invertible on \(X\setminus Z\).
  For each~\(i\),
  we define \(a_{i}\colon X\to A\)
  as in \cref{e:a_from_x}.
  Each \(a_{i}\) is continuous on \(X\setminus Z\).
It is also continuous at any \(p\in Z\),
  since for~\(\epsilon>0\),
  the image of
  any element
  in \(y_{i}^{-1}(A_{<\epsilon})\)
  under~\(a_{i}\)
  has norm \(<\sqrt{\epsilon}\).
  These constitute a solution to the equations.

  Lastly, we suppose~\cref{i:n_nab} of \cref{exc_na}.
  It suffices to observe that
  the same proof as in the previous case works;
  i.e., each construction does not leave the unit ball.
  First, \(x\) lands in the unit ball
  since \(v\) lands in \(\NN\cup\{\infty\}\).
  For \(a_{i}\),
  the norm estimate we use to show its continuity
  at \(p\in Z\)
  also shows that
  it lands in the unit ball.
\end{proof}

\begin{proof}[Proof of \cref{exc_a}]
  By \cref{f37c1d2136,26b5462836},
  it suffices to show that
  for any closed inclusion \(Z\hookrightarrow X\)
  of compacta of countable weight,
  the map
  \begin{equation}
    \label{e:x678f}
    \Cls{C}(Z;A)\otimes_{\Cls{C}(X;A)}\Cls{C}(Z;A)
    \to
    \Cls{C}(Z;A)
  \end{equation}
  is an equivalence.
  This follows from \cref{b485ea33ba,xxaaq2}.
\end{proof}

\begin{proof}[Proof of \cref{exc_na}]
  By \cref{xds8fz,26b5462836},
  it suffices to show that
  for any closed inclusion \(Z\hookrightarrow X\)
  of profinite sets of countable weight,
  \cref{e:x678f} is an equivalence.
  This follows from \cref{7d5f1b5f38,xxaaq2}.
\end{proof}

\section{Preparation for the noncommutative case}\label{s:nc}

The goal of this technical section
is to give a slight extension of
the following vanishing results,
proven in~\cite[Proposition~5]{KerzStrunk17}
and~\cite[Proposition~6.4]{KerzStrunkTamme18},
respectively:

\begin{theorem}[Kerz--Strunk]\label{ks}
  Let \(X\) be a reduced scheme
  which is quasiprojective over a noetherian commutative ring.
  Let \(f\colon Y\to X\) be
  a smooth and quasiprojective morphism.
  For any \(\alpha\in K_{*}(Y)\) with \({*}<0\),
  there exists a projective birational morphism
  \(p\colon X'\to X\)
  such that \(\alpha\) vanishes in \(K_*(Y')\).
\end{theorem}

\begin{theorem}[Kerz--Strunk--Tamme]\label{kst}
  In the situation of \cref{ks},
  for any \(\alpha\in\NK_*(Y)\)\footnote{Recall that
    \(\NK(\X)\) is
    the complement of
    the direct summand~\(K(\X)\)
    of~\(K(\AA_{\X}^1)\).
  },
  there exists a projective birational morphism
  \(p\colon X'\to X\)
  such that \(\alpha\) vanishes in \(\NK_*(Y')\).
\end{theorem}

Our generalization is only needed to prove
our main theorems
when the local division ring~\(F\)
is noncommutative\footnote{When \(F=\HH\),
  we can use the resolution of singularities
  to avoid this.
}.
It concerns
the following situation:

\begin{assumption}\label{finsep}
  Let \(R\) be a commutative ring
  and \(S\) an associative \(R\)-algebra,
  i.e., an associative algebra
  with an algebra map \(R\to S\)
  whose image is contained in the center.
  We assume that
  it is finitely generated
  projective as an \(R\)-module
  and
  the multiplication
  \(S\otimes_{R}S\to S\) admits
  an \((S,S)\)-bilinear section.
\end{assumption}

\begin{remark}\label{xdfljs}
  \Cref{finsep} can be thought of as
  an additive counterpart
  of the notion
  of a saturated (aka smooth proper) \(\E_{1}\)-algebras:
  Those conditions require that
  over both~\(R\)
  and the enveloping algebra,
  \(S\) be finitely generated projective 
  instead of perfect.
  Hence \cref{finsep} implies
  that \(S\) is saturated over~\(R\)
  but not vice versa;
consider
  the algebra of upper-triangular \(2\times2\) matrices.
\end{remark}

\begin{remark}\label{1d3a23503d}
  In \cref{finsep},
  the finiteness hypothesis is redundant
  by~\cite[Theorem~1]{VillamayorZelinsky66}.
\end{remark}

\begin{example}\label{b2dea3b9aa}
  When \(R\) is a field,
  \cref{finsep} is satisfied if and only if
  \(S\) is a finite product of matrix algebras
  over finite-dimensional division algebras
  whose centers are finite-dimensional separable field extensions
  of~\(R\).
\end{example}

\begin{example}\label{xpfq5y}
Recall that a \emph{local division ring}
  is a topological division ring
  whose
  underlying topological space is a local compactum
  but not discrete.
  It is classical that
  it is a finite-dimensional division algebra
  over
  \(\FF_{p}\llparenthesis T\rrparenthesis\),
  \(\QQ_{p}\),
  or~\(\RR\)
  for some prime~\(p\).
  Then
  as an algebra over its center,
  it satisfies \cref{finsep}.
\end{example}

For an associative algebra,
we write \(\D(\X)\)
for the derived \(\infty\)-category of 
right modules.
We introduce the following notation:

\begin{definition}\label{a4bae89270}
  Under \cref{finsep},
  for a quasicompact quasiseparated
  \(R\)-scheme~\(X\),
  we define
  \(K(X_S)\) to be
  the \(K\)-theory
  of \((\D(X)\otimes_{\D(R)}\D(S))^{\aleph_{0}}\).
\end{definition}

The main result of this section is the following:

\begin{theorem}\label{xpbwff}
  Under \cref{finsep},
  in the situation of \cref{ks},
  consider a map \(X\to\Spec R\).
  \begin{enumerate}
    \item\label{i:ks}
      Then for any \({*}<0\) and \(\alpha\in K_{*}(Y_{S})\),
      there is a birational projective morphism
      \(p\colon X'\to X\) such that
      \(\alpha\) vanishes in \(K_{*}(Y'_{S})\).
    \item\label{i:kst}
      Then for any \(\alpha\in\NK_{*}(Y_{S})\),
      there is a birational projective morphism
      \(p\colon X'\to X\) such that
      \(\alpha\) vanishes in \(\NK_{*}(Y'_{S})\).
  \end{enumerate}
\end{theorem}

\begin{proof}
  We first recall
  the proofs of \cref{ks,kst}
  and then explain
  how to modify the argument
  in order to get the desired variants.
  Note that our proof of \cref{ks},
  which we explain here, is a bit
  different from the original one;
  we rather follow the strategy
  of the original proof of \cref{kst}.

  Consider the limit space
  \(Y''=\projlim_{X'\to X}Y'\)
  where \(X'\to X\) runs
  over birational projective morphisms.
  It suffices to show
  the vanishing of \(K_*(Y'')\) for \({*}<0\)
  and~\(\NK(Y'')\).
  Since \(Y'\) admits an ample family
  of line bundles,
  \(K(Y')\) is equivalent to
  the \(K\)-theory of
  the exact category\footnote{By ``exact category'',
    we mean a Quillen exact \(1\)-category.
  }~\(\Vec(X)\).
  Also, \(\NK(Y')\) is the suspension
  of the complement
  of the direct summand \(K(\Vec(Y'))\) of~\(K(\NVec(Y'))\),
  where \(\NVec(Y')\) denotes
  the exact category of vector bundles
  with nilpotent endomorphisms.
  Therefore,
  we need to show that
  \(K(\Vec(Y''))\) is connective and
  \(K(\Vec(Y''))\to K(\NVec(Y''))\) is an equivalence,
  where 
  \(Y''\) is
  considered as a locally ringed space.
  It is proven
  in the proof of~\cite[Proposition~6.4]{KerzStrunkTamme18}
  that \(Y''\) has a coherent structure sheaf.
  We then consider
  the following diagrams
  of exact \(1\)-categories:
  \begin{equation}
    \label{e:27sdz}
    \begin{tikzcd}
      \Vec(Y'')\ar[r,hook]\ar[d,hook]&
      \NVec(Y'')\ar[d,hook]\\
      \Coh(Y'')\ar[r,hook]&
      \NCoh(Y'')\rlap.
    \end{tikzcd}
  \end{equation}
  Here the bottom row is the abelian categories
  of coherent sheaves,
  and those with nilpotent endomorphisms,
  respectively.
  In the proof of~\cite[Proposition~6.4]{KerzStrunkTamme18},
  it is shown that
  the vertical arrows satisfy
  the assumption of the resolution theorem
  and thus induce equivalences on~\(K\).
  By~\cite[Remark~4.2.3]{FujiwaraKato18},
  the functor \(\injlim_{X'\to X}\Coh(Y')\to\Coh(Y'')\)
  is an equivalence.
  The same holds for \(\NCoh\)
  as it is the kernel of
  \(\Coh(\AA^{1}_{\X})\to\Coh(\GG_{m,\X})\).
  Hence \(K\) of the bottom terms is connective
  by the connectivity theorem of Schlichting~\cite[Theorem~7]{Schlichting06},
  from which we get \cref{ks}.
  By the Quillen dévissage theorem,
  the bottom row is an equivalence,
  from which we also get \cref{kst}.

  Then we get back to our variant.
  For an idempotent-complete additive \(R\)-linear \(\infty\)-category~\(\cat{A}\),
  we write \(\cat{A}\otimes_{R}S\)
  for the \(\infty\)-category
  of compact objects of
  \(\PShv_{\Sigma}(\cat{A})
  \otimes_{\D(R)_{\geq0}}\D(S)_{\geq0}\);
  cf.~\cite[Sections~D.1--2]{LurieSAG}.
  With this notation,
  \(K(Y'_{S})\) is by definition \(K(\Perf(Y')\otimes_{R}S)\).
  We apply \(\X\otimes_{R}S\) to 
  the diagram~\cref{e:27sdz}.
  It is easy to observe that
  for an idempotent-complete \(R\)-linear
  exact category~\(\cat{E}\),
  its base change~\(\cat{E}\otimes_{R}S\)
  inherits a structure of an exact category.
  We conclude
  the desired result from \cref{xarrml},
  which we prove below.
\end{proof}

\begin{lemma}\label{xarrml}
Under \cref{finsep},
  consider
  an idempotent-complete \(R\)-linear exact category~\(\cat{E}\).
  Then the canonical functor
  \(
  F\colon
  \Db(\cat{E}\otimes_{R}S)
  \to
  \Db(\cat{E})\otimes_{R}S
  \)
  is an equivalence.
\end{lemma}

\begin{proof}
  Note that
  by~\cite[Theorem~2.8]{BalmerSchlichting01}
  and~\cite[Lemma~1.2.4.6]{LurieHA},
  if an exact category~\(\cat{E}\) is idempotent complete,
  so is \(\Db(\cat{E})\).

  We first show that
  it is essentially surjective.
  Let \(D\) be an object
  of \(\Db(\cat{E})\otimes_{R}S\).
  As its image is idempotent complete
  and stable,
  without loss of generality,
  we can assume that
  \(D\) is of the form \(E\otimes_{R}S\)
  for an object~\(E\) of
  \(\Db(\cat{E})\).
  This clearly can be
  represented as a complex
  of \(S\)-modules in~\(\cat{E}\).

  We then show that
  it is fully faithful.
  We prove that
  \(\map(C,C')\to\map(FC,FC')\)
  is an equivalence
  for any \(C\) and~\(C'\),
  where \(\map\) denotes the mapping spectrum.
  By the same argument,
  without loss of generality,
  we can assume
  that both are images
  in \(\cat{E}\to\Db(\cat{E})\),
  but in this case,
  the assertion is clear.
\end{proof}

\section{\texorpdfstring{\(K\)}{K}-theory of rings of continuous functions}\label{s:main}

In this section, we demonstrate the main results of this paper.
Building on axiomatic arguments in \cref{ss:ha},
we prove generalized versions
of our main theorems
in \cref{ss:main}.
We conclude with the variants
for real compacta
and
with order coefficients
in \cref{ss:kr,ss:order},
respectively.

\subsection{Some homological algebra of sheaves}\label{ss:ha}

We here observe some abstract properties of cd~sheaves
under the situation
where we know that
certain classes vanish cd~locally.
To run the induction,
the following notion is convenient:

\begin{definition}\label{1684f2d18a}
  Let \(F\colon((\Cat{Cpt}_{\aleph_{1}})_{/S})^{\op}\to\Cat{Sp}\)
  be a cd~sheaf over \(S\in\Cat{Cpt}_{\aleph_{1}}\).
  For an integer \(n\geq0\),
  we say that a class
  \(\alpha\in\pi_{*}F(X)\)
  has \emph{complexity}~\(\leq n\),
  if there is a cd~cover
  \(\{Y_{i}\to X\}_{i}\)
  of complexity \(\leq n\)
  (see \cref{c0f0e840c9})
  such that
  \(\alpha\) is killed in \(\pi_{*}F(Y_{i})\)
  for each~\(i\).
  We say that a class
  has \emph{finite complexity}
  if it has complexity~\(\leq n\)
  for some~\(n\).

  We define a similar notion
  when \(\Cat{Cpt}_{\aleph_{1}}\)
  is replaced by \(\Cat{PFin}_{\aleph_{1}}\).
\end{definition}

\begin{lemma}\label{xv5iy7}
  In the situation of \cref{1684f2d18a},
  suppose that
  a class \(\alpha\in\pi_{*}F(X)\) has complexity \(\leq n\)
  for \(n\geq1\).
  Then there is a blowup square~\cref{e:blowup}
  such that
  the image of \(\alpha\) is zero in \(\pi_{*}F(X')\)
  and has complexity \(\leq n-1\) in \(\pi_{*}F(Z)\).
\end{lemma}

\begin{proof}
  In the notation of \cref{c0f0e840c9},
  the blowup square determined by
  \(X'=X_{n}'\) and \(Z=X_{n-1}\)
  satisfies the desired property.
\end{proof}

In the profinite case,
we use the following:

\begin{proposition}\label{f59fc7068a}
  Let \(F\colon((\Cat{PFin}_{\aleph_{1}})_{/S})^{\op}\to\Cat{Sp}\)
  be a cd~sheaf over \(\Cat{PFin}_{\aleph_{1}}\).
  If any class of negative degree
  has finite complexity,
  then \(F\) is connective.
\end{proposition}

\begin{proof}
  For any blowup square~\cref{e:blowup}
  of profinite sets of countable weight,
  by \cref{xds8fz},
  we have a long exact sequence
  \begin{equation*}
    \dotsb
    \to\pi_*(F(X'))
    \to\pi_*(F(X))
    \oplus\pi_*(F(Z'))
    \to\pi_*(F(Z))
    \to\dotsb.
  \end{equation*}
  We see that this splits into
  the short exact sequences
  \begin{equation*}
    0
    \to\pi_*(F(X'))
    \to\pi_*(F(X))
    \oplus\pi_*(F(Z'))
    \to\pi_*(F(Z))
    \to0
  \end{equation*}
  as follows:
  When \(Z=\emptyset\), this is obvious.
  Otherwise,
  by the theorem of Halmos~\cite{Halmos61},
  we have a retract \(X\to Z\),
  which gives us the desired splitting.
  From this,
  by induction on the complexity
  and \cref{xv5iy7},
  we see that
  every class of negative degree is zero.
\end{proof}

In the compactum case,
we prove the following two:

\begin{proposition}\label{8b58168a08}
  Let \(F\colon((\Cat{Cpt}_{\aleph_{1}})_{/S})^{\op}\to\Cat{Sp}\)
  be a cd~sheaf over \(S\in\Cat{Cpt}_{\aleph_{1}}\).
  Assume that any class of negative degree
  has finite complexity.
  Then the connective cover map
  \(\tau_{\geq0}F\to F\)
  of presheaves is the cd~sheafification morphism.
\end{proposition}

\begin{proof}
Let \(E\) be the sheafification of \(\tau_{\geq0}F\).
  We have a canonical map \(E\to F\).
  We need to prove that
  this is an equivalence.
  Since its connective cover is an equivalence,
  it suffices to check that
  this is an equivalence on each point
  of the cd~topos of~\(S\).
  Since both \(E\) and~\(F\) are connective on each point
  by assumption,
  the desired result follows.
\end{proof}

\begin{proposition}\label{a2ca91099a}
  Let \(S\) be a compactum of countable weight
  and \(E\to F\)
  a map of \(\Cat{Sp}\)-valued cd~sheaves on
  \((\Cat{Cpt}_{\aleph_{1}})_{/S}\).
  Suppose
  for any \(X\in(\Cat{Cpt}_{\aleph_{1}})_{/S}\)
  that
  \(E(X)\to F(X)\) is bijective on~\(\pi_{0}\)
  and each class in
  \(E_*(X)\) or \(F_*(X)\)
  for \({*}<0\) has finite complexity.
  Then
  \(\tau_{\leq0}E(X)\to\tau_{\leq0}F(X)\) is an equivalence
  for any \(X\in(\Cat{Cpt}_{\aleph_{1}})_{/S}\).
\end{proposition}

\begin{proof}
  This follows from induction
  and \cref{891d38c320} below.
\end{proof}

\begin{lemma}\label{891d38c320}
  Let \(E\to F\)
  be a map of \(\Cat{Sp}\)-valued cd~sheaves on
  \(\Cat{Cpt}_{\aleph_{1}}\).
  Suppose that
  \(E(X)\to F(X)\) is surjective on~\(\pi_{0}\)
  for each~\(X\).
  \begin{enumerate}
    \item\label{i:tan}
      Suppose that
      any class in
      \(\pi_{-1}E(X)\)
      has finite complexity for any~\(X\)
      and \(E(X)\to F(X)\) is injective on~\(\pi_{0}\).
      Then
      the map
      \(E(X)\to F(X)\) is injective on~\(\pi_{-1}\)
      for each~\(X\).
    \item\label{i:zen}
      Suppose that
      any class in
      \(\pi_{-1}F(X)\)
      has finite complexity for any~\(X\)
      and \(E(X)\to F(X)\) is injective on~\(\pi_{-1}\).
      Then
      the map
      \(E(X)\to F(X)\) is surjective on~\(\pi_{-1}\)
      for each~\(X\).
  \end{enumerate}
\end{lemma}

\begin{proof}
  In this proof,
  we write
  \(E_*(\X)\) and \(F_*(\X)\)
  for \(\pi_*E(\X)\) and \(\pi_*F(\X)\),
  respectively.
  Note that by \cref{f37c1d2136},
  we have an exact sequence
  \begin{equation*}
    \dotsb
    \to E_{*}(X')
    \to E_{*}(X)\oplus E_{*}(Z')
    \to E_{*}(Z)
    \to\dotsb
  \end{equation*}
  for any blowup square \cref{e:blowup}
  and similarly for~\(F\).

  We prove \cref{i:tan}.
  We take a class
  \(\alpha\in E_{-1}(X)\)
  that is killed in \(F_{-1}(X)\).
  We wish to show that \(\alpha\) is zero.
  We proceed by induction on the complexity of~\(\alpha\).
  If the complexity is \(\leq0\), it is zero by definition.
  Fix \(n\geq1\).
  Assume that
  the claim holds for the complexity \(\leq n-1\)
  and
  that \(\alpha\) has complexity \(\leq n\).
  By \cref{xv5iy7},
  we can take a blowup square \cref{e:blowup}
  such that
  the image of~\(\alpha\)
  is zero
  in \(E_{-1}(X')\)
  and
  has complexity \(\leq n-1\)
  in \(E_{-1}(Z)\).
  We consider
  the map of long exact sequences
  \begin{equation*}
    \begin{tikzcd}
      \dotsb\ar[r]&
      E_{0}(X')\oplus E_{0}(Z)\ar[r]\ar[d,tail,two heads]&
      E_{0}(Z')\ar[r]\ar[d,tail,two heads]&
      E_{-1}(X)\ar[r]\ar[d]&
      E_{-1}(X')\oplus E_{-1}(Z)\ar[r]\ar[d]&
      \dotsb\\
      \dotsb\ar[r]&
      F_{0}(X')\oplus F_{0}(Z)\ar[r]&
      F_{0}(Z')\ar[r]&
      F_{-1}(X)\ar[r]&
      F_{-1}(X')\oplus F_{-1}(Z)\ar[r]&
      \dotsb\rlap.
    \end{tikzcd}
  \end{equation*}
  By the inductive hypothesis, both images are zero.
  Hence \(\alpha\) is the image
  of a class \(\gamma'\in E_{0}(Z')\).
  As \(\alpha\) is mapped to zero in \(F_{-1}(X)\),
  the image of~\(\gamma'\) in \(F_{0}(Z')\)
  is in the image of the map
  \(F_{0}(X')\oplus F_{0}(Z)\to F_{0}(Z')\).
  However, by assumption, this means that
  \(\gamma'\) is in the image of
  \(E_{0}(X')\oplus E_{0}(Z)\to E_{0}(Z')\)
  and therefore \(\alpha\) is zero.

  We then prove \cref{i:zen}.
  We take \(\beta\in F_{-1}(X)\).
  We have to show that \(\beta\) is the image
  of some class in \(E_{-1}(X)\).
  We proceed by induction on the complexity of~\(\beta\).
  If the complexity is \(\leq0\), it is clear.
  Fix \(n\geq1\) and
  suppose that
  the claim is true for the complexity \(\leq n-1\)
  and
  \(\beta\) has complexity \(\leq n\).
  By \cref{xv5iy7},
  we can take a blowup square \cref{e:blowup}
  such that
  \(\beta\) is killed in \(F_{-1}(X')\)
  and its image \(\delta\in F_{-1}(Z)\)
  has complexity \(\leq n-1\).
  We consider
  the map of long exact sequences
  \begin{equation*}
    \begin{tikzcd}
      \dotsb\ar[r]&
      E_{0}(Z')\ar[r]\ar[d,two heads]&
      E_{-1}(X)\ar[r]\ar[d,tail]&
      E_{-1}(X')\oplus E_{-1}(Z)\ar[r]\ar[d,tail]&
      E_{-1}(Z')\ar[r]\ar[d,tail]&
      \dotsb\\
      \dotsb\ar[r]&
      F_{0}(Z')\ar[r]&
      F_{-1}(X)\ar[r]&
      F_{-1}(X')\oplus F_{-1}(Z)\ar[r]&
      F_{-1}(Z')\ar[r]&
      \dotsb\rlap.
    \end{tikzcd}
  \end{equation*}
  By the inductive hypothesis,
  we can take
  \(\gamma\in E_{-1}(Z)\)
  that is mapped to \(\delta\).
  Since \(E_{-1}(Z')\to F_{-1}(Z')\) is injective,
  \(\gamma\) is killed in \(E_{-1}(Z')\).
Therefore, we have a class
  \(\tilde\alpha\in E_{-1}(X)\)
  that is mapped to~\((0,\gamma)\).
  We write \(\tilde\beta\) for its image
  in \(F_{-1}(X)\).
  Then \(\beta-\tilde\beta\)
  is killed by
  \(F_{-1}(X)\oplus F_{-1}(X')\oplus F_{-1}(Z)\).
  Hence we can take \(\delta'\in F_{0}(Z')\)
  that is mapped to this class.
  By the surjectivity of \(E_{0}(Z')\to F_{0}(Z')\),
  we get \(\gamma'\in E_{0}(Z')\)
  that is mapped to~\(\delta'\).
  The sum of \(\tilde\alpha\)
  and the image of \(\gamma'\) in \(E_{-1}(X)\)
  is mapped to \(\beta\in F_{-1}(X)\).
\end{proof}

\subsection{\texorpdfstring{\(K\)}{K}-theory of rings of continuous functions}\label{ss:main}

In this section,
we prove \cref{main_a} and the following,
which generalizes \cref{main_na,main_hi}:

\begin{theorem}\label{x8biir}
  Let \(X\) be a compactum and
  \(F\) a local division ring with center~\(Z\).
  Suppose that \(A\) is a commutative ring
  smooth over \(\Cls{C}(X;Z)\).
  \begin{enumerate}
    \item\label{i:m_neg}
      When \(F\) is nonarchimedean,
      \(K(A\otimes_{Z}F)\)
      is connective.
    \item\label{i:m_reg}
      The ring
      \(A\otimes_{Z}F\)
      is \(K\)-regular.
  \end{enumerate}
\end{theorem}

The main input of our arguments
is the following:

\begin{proposition}\label{vain}
  In the situation of \cref{x8biir},
  we have the following:
  \begin{enumerate}
    \item\label{i:ygi2s}
      For any class
      \(\alpha\in K_{*}(A\otimes_{Z}F)\)
      for \({*}<0\),
      there is a cd~cover \(Y\to X\)
      of compacta
      such that \(\alpha\)
      vanishes in
      \(K_{*}(A\otimes_{\Cls{C}(X;Z)}\Cls{C}(Y;Z))\).
    \item\label{i:4qz49}
      For any class
      \(\alpha\in\NK_{*}(A\otimes_{Z}F)\),
      there is a cd~cover \(Y\to X\)
      of compacta
      such that \(\alpha\)
      vanishes in
      \(\NK_{*}(A\otimes_{\Cls{C}(X;Z)}\Cls{C}(Y;F))\).
  \end{enumerate}
\end{proposition}

\begin{proof}
  We prove~\cref{i:ygi2s}.
  We write \(\Cls{C}(X;Z)\) as
  a filtered colimit of \(Z\)-algebras of finite type
  as \(\injlim_{i}C_{i}\).
  Then there is~\(i\) and a smooth algebra \(A_{i}\) over~\(C_{i}\)
  such that \(A\) is
  the base change of~\(A_{i}\) along \(C_{i}\to\Cls{C}(X;Z)\).
Since \(K\) commutes with filtered colimits,
  we can retake~\(i\) bigger if necessary so that 
  \(\alpha\) comes from
  a class in \(K_*(A_{i}\otimes_{Z}F)\).
  By using~\cref{i:ks} of \cref{xpbwff}
  and induction on the dimension of~\(\Spec C_{i}\),
  we get an affine rh~cover
  \(\Spec D\to\Spec C_{i}\)
  such that 
  the class vanishes in
  \(K_*(A\otimes_{C_{i}}D\otimes_{Z}F)\).
  Let \(Y\) be the base change
  of \((\Spec D)(F)\to(\Spec C_{i})(F)\)
  along \(X\to(\Spec C_{i})(F)\).
  Then by \cref{cd_rh},
  \(Y\to X\) is a cd~cover
  satisfying the desired condition.

  Similarly,
  we obtain \cref{i:4qz49}
  by using~\cref{i:kst},
  instead of~\cref{i:ks},
  of \cref{xpbwff}.
\end{proof}

\begin{proof}[Proof of \cref{main_a}]
  We consider the quaternionic case,
  as the other cases are similar.
  We have canonical maps of spectra
  \begin{equation*}
    \tau_{\geq0}K(\Cls{C}(X;\HH))
    \to
    \tau_{\geq0}\Gamma(X;\cst{\KSp})
    \gets
    \tau_{\geq0}\Gamma(X;\cst{\ksp})
  \end{equation*}
  for a compactum~\(X\),
  which induces an equivalence on \(\pi_{0}\).
  The first map is classical;
  e.g.,
  we will construct this in~\cite{k-ros-2}.
  The second map is an equivalence.
  Hence we have a map
  \(
  \tau_{\geq0}K(\Cls{C}(\X;\HH))
  \to
  \tau_{\geq0}\Gamma(\X;\cst{\ksp})
  \)
  of \(\Cat{Sp}\)-valued presheaves
  on \(\Cat{Cpt}_{\aleph_{1}}\).
  Since
  \(K(\Cls{C}(\X;\HH))\)
  and
  \(\Gamma(\X;\cst{\ksp})\)
  are sheaves
  by \cref{exc_a,f1a75190e6},
  respectively,
  its sheafification gives us a morphism
  \(K(\Cls{C}(\X;\HH))\to\Gamma(\X;\cst{\ksp})\)
  by \cref{8b58168a08}.
  Then we can apply \cref{a2ca91099a}
  to get the desired result,
  since
  the complexity conditions for
  them
  and
  \(\Gamma(\X;\cst{\ksp})\)
  follow from~\cref{i:ygi2s}
  of \cref{vain}
and \cref{zero_zero},
  respectively.
\end{proof}

\begin{proof}[Proof of \cref{i:m_neg} of \cref{x8biir}]
  By \cref{exc_na},
  the functor
  \(K(A\otimes_{\Cls{C}(X;Z)}\Cls{C}(\X;Z)\otimes_{Z}F)\colon
  (\Cat{PFin}_{\aleph_{1}})_{/X}^{\op}
  \to\Cat{Sp}\) is a cd~sheaf.
  Therefore,
  in light of \cref{f59fc7068a},
  the desired result follows from~\cref{i:ygi2s}
  of \cref{vain}.
\end{proof}

Now
only \cref{i:m_reg} of \cref{x8biir}
is left to be proven.
One might want to prove this
as a corollary of~\cref{i:4qz49} of \cref{vain}.
However, it requires some care:
We know that \(\KH(\Cls{C}(\X;F))\)
is a cd~sheaf whose
stalks vanish,
but as the cd~site \(\Cat{Cpt}_{\aleph_{1}}\)
(or its minimal variant in \cref{xah2cq}\footnote{Note that it is also nontrivial that
  it suffices to consider this case;
  nevertheless,
  it follows from
  Rosenberg's argument on~\cite[page~91]{Rosenberg97},
}) is not hypercomplete,
it does not imply
the vanishing of itself.
Nevertheless,
we can rely on another hypercompleteness statement:

\begin{proof}[Proof of \cref{i:m_reg} of \cref{x8biir}]
  We only treat the archimedean case as the nonarchimedean case is similar.
  We write \(\Cls{C}(X;Z)\) as
  a filtered colimit of \(Z\)-algebras of finite type
  as \(\injlim_{i}C_{i}\).
  Then there is~\(i\) and a smooth algebra \(A_{i}\) over~\(C_{i}\)
  such that \(A\) is
  the base change of~\(A_{i}\) along \(C_{i}\to\Cls{C}(X;Z)\).
  Then we consider the presheaf
  \(G=\NK(A_{i}\otimes_{C_{i}}\X\otimes_{Z}F)\)
  on the cdh~site \(\Cat{Aff}_{C_{i}}
  =(\Cat{Aff}_{Z})_{/\Spec C_{i}}\)
  of affine schemes of finite type over~\(C_{i}\).
  Since its stalks are zero by \cref{i:kst} of \cref{xpbwff}
  its sheafification~\(G'\) is zero;
  here we use the hypercompleteness of the site
  proven in~\cite[Proposition~2.11]{Voevodsky10M}
  (see~\cite{EHIK21C} for a more general result).
  Since the functor
  \(\NK(A\otimes_{\Cls{C}(X;Z)}\Cls{C}(\X;Z)\otimes_{Z}F)\colon
  (\Cat{Cpt}_{\aleph_{1}})_{/X}^{\op}
  \to\Cat{Sp}\) is a cd~sheaf by \cref{exc_a},
the desired result follows from \cref{xwb4zy}.
\end{proof}

\subsection{Variant: commutative real C*-algebras}\label{ss:kr}

In this section,
we compute the negative \(K\)-theory of commutative real C*-algebras.
We start by recalling a classical result,
which was essentially proven in~\cite{ArensKaplansky48}:

\begin{definition}\label{50b22568f4}
  A \emph{real\footnote{This usage of the word ``real'' is due to Atiyah~\cite{Atiyah66}.
  } compactum}~\((X,\tau)\)
  is a compactum with an involution.
  We write \(\Cat{RCpt}\)
  for the category of real compacta.
  For such,
  we write \(\Cls{C}(X,\tau)\) for
  the subring of \(\Cls{C}(X;\CC)\)
  spanned by functions~\(f\)
  satisfying \(f\circ\tau=\overline{(\X)}\circ f\).
\end{definition}

\begin{example}\label{8ea6072d0a}
  For a compactum~\(X\),
  the algebra
  \(\Cls{C}(X,{\id})\)
  is identified with
  \(\Cls{C}(X;\RR)\).
\end{example}

\begin{example}\label{98d62ca7a8}
  For a compactum~\(X\),
  consider its double \(X\amalg X\)
  and the involution~\(\sw\)
  given by swapping the two copies of~\(X\).
  Then the algebra
  \(\Cls{C}(X\amalg X,{\sw})\)
  is identified with
  \(\Cls{C}(X;\CC)\).
\end{example}

\begin{theorem}[Arens--Kaplansky]\label{57ed4323f9}
  The assignment \((X,\tau)\mapsto\Cls{C}(X,\tau)\)
  gives an equivalence
  from \(\Cat{RCpt}^{\op}\)
  to the category
  of commutative real C*-algebras.
\end{theorem}

We prove the following variant of \cref{main_a}:

\begin{theorem}\label{main_kr}
  For a real compactum~\((X,\tau)\),
  there is an equivalence
  \begin{equation*}
    \tau_{\leq0}
    K(\Cls{C}(X,\tau))
    \simeq
    \tau_{\leq0}
    (\Gamma(X;\cst{\ku})^{hC_{2}})
  \end{equation*}
  where the \(C_{2}\)-action
  comes from~\(\tau\) on~\(X\)
  and the standard \(C_{2}\)-action on~\(\ku\).
\end{theorem}

\begin{remark}\label{xqhxql}
  For a real compactum \((X,\tau)\),
  the homotopy group
  \(\pi_{*}(\Gamma((X,\tau);\cst{\ku})^{hC_{2}})\)
  is commonly written as \(\kr^{-{*}}(X,\tau)\).
\end{remark}

We prove \cref{main_kr}
by reducing to the real and complex cases:

\begin{proof}
By arguing as in \cref{main_a},
  we can obtain a map \(K(\Cls{C}(\X;\CC))\to\Gamma(\X;\cst{\ku})\)
  between \(\Fun(BC_{2},\Cat{Sp})\)-valued presheaves
  on \(\Cat{Cpt}_{\aleph_{1}}\).
  Hence we have a map
  \begin{equation*}
    \tau_{\leq0}
    K(\Cls{C}(X,\tau))
    \to
    \tau_{\leq0}
    (\Gamma(X;\cst{\ku})^{hC_{2}})
  \end{equation*}
  functorial in \((X,\tau)\).
  We prove that this is an equivalence.
  Note that
  we know the cases of \cref{8ea6072d0a,98d62ca7a8}
  by \cref{main_a}.

  For a general real compactum~\((X,\tau)\),
  we choose an open subset~\(U\) such that
  \(X\setminus X^{C_{2}}=U\amalg\tau(U)\) holds.
  Then
  the square
  \begin{equation*}
    \begin{tikzcd}
      X^{C_{2}}\ar[r]\ar[d]&
      X\ar[d]\\
      \{\infty\}\ar[r]&
      (U\amalg\tau(U))^+\rlap,
    \end{tikzcd}
  \end{equation*}
  of real compacta is carried to a cartesian square
  by \(K(\Cls{C}(\X))\):
  This follows from \cref{ed50fc2cda,3234ba150e,xxaaq2}
  and the observation that
  the kernels of
  \(\cat{C}((U\amalg\tau(U))^{+})\to\cat{C}(\{\infty\})\)
  and
  \(\cat{C}(U^{+};\CC)\to\cat{C}(\{\infty\};\CC)\)
  are isomorphic as nonunital rings;
  we can also use \cref{994969f6af} to deduce this.
  Therefore, we are reduced to the case of \((U\amalg\tau(U))^+\).
  By considering the square
  \begin{equation*}
    \begin{tikzcd}
      \{\infty\}\amalg\{\infty\}\ar[r]\ar[d]&
      U^+\amalg\tau(U)^+\ar[d]\\
      \{\infty\}\ar[r]&
      (U\amalg\tau(U))^+
    \end{tikzcd}
  \end{equation*}
  and the same excision result,
  we obtain the desired equivalence.
\end{proof}

The following analog of \cref{main_hi} is obtained similarly:

\begin{theorem}\label{xz0crg}
  Any commutative real C*-algebra is \(K\)-regular.
\end{theorem}

\begin{proof}
  By arguing as in the proof of \cref{main_kr},
  but applying \cref{26b5462836}
  instead of \cref{3234ba150e},
  the desired result follows from \cref{main_hi}.
\end{proof}

\subsection{Variant: rings of integers}\label{ss:order}

In this section, we prove the following:

\begin{theorem}\label{main_int}
  Let \(O\) be the ring of integers
  of a nonarchimedean local field~\(F\)
  and \(X\) a profinite set.
  Suppose that a commutative ring~\(A\)
  is smooth over \(\Cls{C}(X;O)\).
  Then the negative \(K\)-theory of \(A\) vanishes
  and \(A\) is \(K\)-regular.
\end{theorem}

Our strategy is simple:
We deduce this from
the cases of 
the fraction and residue fields of~\(O\).

\begin{proposition}\label{x89}
  Let \(k\) be a discrete field
  and \(X\) a profinite set.
  Then \(\Cls{C}(X;k)\) is \(\Ind\)-smooth (see \cref{ind_sm}) over~\(k\).
\end{proposition}

\begin{proof}
  We write~\(X\)
  as a cofiltered limit of
  finite sets \(\projlim_{i}X_{i}\).
  Then \(\Cls{C}(X;k)\)
  is isomorphic to \(\injlim_{i}\Cls{C}(X_{i};k)\).
\end{proof}

From this, we see the following:

\begin{corollary}\label{k_dis}
  Every smooth commutative algebra over \(\Cls{C}(X;k)\)
  is \(K\)-regular
  and its negative \(K\)-theory vanishes.
\end{corollary}

We need a bit stronger statement to prove \cref{main_int}.
First, we recall the following:

\begin{definition}\label{x5l9jg}
  We call a coherent ring \(R\) \emph{regular}
  if every finitely presented module
  has finite projective dimension.
\end{definition}

\begin{example}\label{xrd00s}
  For a field~\(k\),
  the polynomial ring \(k[T_{i}\mid i\in I]\)
  is regular coherent for any set~\(I\).
\end{example}

\begin{proposition}\label{dis}
  In the situation of \cref{k_dis},
  suppose that \(A\) is a commutative ring
  smooth over \(\Cls{C}(X;k)\).
  Then \(A\) is regular coherent.
\end{proposition}

\begin{proof}
  The proof of \cref{x89}
  shows that \(A\) is not just
  \(\Ind\)-smooth,
  but also it is a filtered colimit of
  smooth algebras with smooth transition maps,
  (cf.~\cref{ind_sm_var}).
  Therefore,
  as a filtered colimit
  of regular noetherian rings with flat transition maps,
  \(A\) is regular coherent;
  see, e.g.,~\cite[Theorem~6.2.2]{Glaz89}.
\end{proof}

\begin{remark}\label{xwquao}
We can deduce \cref{k_dis}
  from \cref{dis}
  since it implies that every smooth algebra
  over \(\Cls{C}(X;k)\) is regular stably coherent.
  Note that when \(F\) is a local field,
  \(\Cls{C}(X;F)\) is rarely coherent:
  Consider \(X=\NN\cup\{\infty\}\)
  and let \(I\) be
  the kernel of the map
  \(\Cls{C}(X;F)\to\Cls{C}(X;F)\)
  determined by the constant endomorphism of~\(X\) at~\(\infty\).
  If \(\Cls{C}(X;F)\) is coherent,
  this is finitely generated.
  Since we know \(I^{2}=I\) from \cref{xxaaq2},
  we see that \(I\) is generated by an idempotent,
  which is not the case.
\end{remark}

\begin{proposition}\label{fibers}
  In the situation of \cref{main_int},
  let \(\FF_q\) denote the residue field of~\(O\).
  Then
  \begin{equation*}
    \Perf(A\otimes_{O}\FF_{q})
    \to
    \Perf(A)
    \to
    \Perf(A\otimes_{O}F),
  \end{equation*}
  where the first map is forgetful,
  induces a fiber sequence on \(K\)-theory.
\end{proposition}

\begin{proof}
  Fix a uniformizer~\(\pi\in O\).
  We need to see that the map
  \begin{equation*}
    \Perf(A/\pi)
    \to
    \ker(
    \Perf(A)
    \to
    \Perf(A[1/\pi])
    )
  \end{equation*}
  induces an equivalence on~\(K\)-theory.
  As the right-hand side 
  is generated by \(A/\pi\),
  this boils down to showing that
  the connective cover
  \(A/\pi\to\REnd_{A}(A/\pi)\) induces an equivalence on \(K\)-theory,
  where \(\REnd\) denotes the derived endomorphism \(\E_{1}\)-algebra.
  This equivalence follows from~\cite[Theorem~1.1]{BurklundLevy23},
  which can be applied here in light of \cref{dis}.
\end{proof}

\begin{proof}[Proof of \cref{main_int}]
  By \cref{fibers},
  we are reduced to the same questions
  on the generic and special fibers,
  which we have treated in \cref{x8biir,k_dis}.
\end{proof}

\let\AA\oldAA \let\Vec\oldVec \bibliographystyle{plain}
 \newcommand{\yyyy}[1]{}

\end{document}